\documentclass[10pt,reqno]{amsart}
\usepackage{mathtools,amssymb,color,eucal,mathrsfs}
\usepackage{mathptmx}
\usepackage{enumitem}
\setitemize{leftmargin=*, itemsep=3pt, topsep=3pt}   
\setenumerate{leftmargin=*, itemsep=3pt, topsep=3pt, 
              label=\textup{(\arabic*)}}             

\usepackage[text={150mm,240mm},centering]{geometry}
\usepackage[noadjust]{cite}

\usepackage{array}
\newcolumntype{L}{>{$\displaystyle}l<{$}}
\newcolumntype{C}{>{$}c<{$}}

\usepackage{multirow}
\usepackage{tikz}
\usepackage{setspace}

\usepackage[colorlinks=true,citecolor=red,linkcolor=blue]{hyperref} 

\allowdisplaybreaks
\numberwithin{equation}{section}

\usepackage[capitalise,noabbrev]{cleveref} 

\newtheorem{thm}{Theorem}[section]
\newtheorem{cor}[thm]{Corollary}
\newtheorem{lemma}[thm]{Lemma}

\newtheorem{prop}[thm]{Proposition}

\newtheorem{result}[thm]{Result}

\newtheorem{theorem}[thm]{Theorem}
\newtheorem{corollary}[thm]{Corollary}
\newtheorem{proposition}[thm]{Proposition}
\newtheorem{conjecture}[thm]{Conjecture}

\newtheorem{assum}{Assumption}

\theoremstyle{definition}

\newtheorem{remark}[thm]{Remark}

\Crefname{thm}{Theorem}{Theorems}
\Crefname{prop}{Proposition}{Propositions}
\Crefname{lemma}{Lemma}{Lemmas}
\Crefname{cor}{Corollary}{Corollaries}
\Crefname{defn}{Definition}{Definitions}
\Crefname{prob}{Problem}{Problems}
\Crefname{conj}{Conjecture}{Conjectures}
\Crefname{assum}{Assumption}{Assumptions}
\Crefname{result}{Result}{Results}

\definecolor{darkgreen}{rgb}{0.1, 0.8, 0.1}

\newcommand{\alg}[1]{\mathfrak{#1}}  
\newcommand{\Mod}[1]{\mathcal{#1}}   
\newcommand{\VOA}[1]{\mathsf{#1}}    
\newcommand{\grp}[1]{\mathsf{#1}}    
\newcommand{\fld}[1]{\mathbb{#1}}    
\newcommand{\lat}[1]{\mathsf{#1}}    
\newcommand{\categ}[1]{\mathscr{#1}} 

\newcommand{\ZZ}{\fld{Z}}
\newcommand{\NN}{\ZZ_{\ge 0}} 

\newcommand{\RR}{\fld{R}}
\newcommand{\CC}{\fld{C}}

\newcommand{\ii}{\mathfrak{i}} 
\newcommand{\ee}{\mathsf{e}}   

\newcommand{\rlat}{\lat{Q}} 

\newcommand{\Com}[2]{\operatorname{Com} ( #1 , #2 )}

\DeclarePairedDelimiter{\brac}{\lparen}{\rparen} 
\DeclarePairedDelimiter{\sqbrac}{\lbrack}{\rbrack} 
\DeclarePairedDelimiter{\set}{\lbrace}{\rbrace}
\newcommand{\st}{\mspace{5mu} : \mspace{5mu}} 
\DeclarePairedDelimiter{\abs}{\lvert}{\rvert}

\DeclarePairedDelimiterX{\comm}[2]{\lbrack}{\rbrack}{#1 , #2}  
\DeclarePairedDelimiterX{\acomm}[2]{\lbrace}{\rbrace}{#1 , #2} 
\DeclarePairedDelimiterX{\inner}[2]{\lparen}{\rparen}{#1 , #2} 
\DeclarePairedDelimiterX{\super}[2]{\lparen}{\rparen}{#1 \delimsize\vert \mathopen{} #2} 

\newcommand{\SLG}[2]{\grp{#1} \brac[\big]{#2}}       

\newcommand{\affine}[1]{\widehat{#1}}
\newcommand{\SLA}[2]{\alg{#1}_{#2}}                       
\newcommand{\SLSA}[3]{\alg{#1} \super*{#2}{#3}}           
\newcommand{\AKMA}[2]{\affine{\alg{#1}}_{#2}}             

\newcommand{\conjaut}{\mathsf{w}}             
\newcommand{\sfaut}{\sigma}                   
\newcommand{\conjmod}[1]{\conjaut(#1)}        
\newcommand{\sfmod}[2]{\sfaut^{#1}(#2)}       

\newcommand{\savoa}[2]{\VOA{L}_{#1}(#2)}  

\newcommand{\slvoa}[1]{\savoa{#1}{\SLA{sl}{2}}}
\newcommand{\eslvoa}[1]{\VOA{E}_{#1}}     

\newcommand{\pfvoa}[1]{\VOA{C}_{#1}}      
\newcommand{\spfvoa}[1]{\VOA{sC}_{#1}}    
\newcommand{\epfvoa}[1]{\VOA{B}_{#1}}     

\newcommand{\virminmod}[1]{\VOA{M}(#1)}       
\newcommand{\nslogminmod}[1]{\VOA{sLM}(#1)}   
\newcommand{\hvoa}{\VOA{H}}                   
\newcommand{\lvoa}[1]{\VOA{V}_{\lat{#1}}}     
\newcommand{\singvoa}[1]{\VOA{S}_{#1}}        
\newcommand{\tripvoa}[1]{\VOA{W}_{#1}}        
\newcommand{\ssingvoa}[1]{\VOA{sS}_{#1}}       
\newcommand{\stripvoa}[1]{\VOA{sW}_{#1}}       

\newcommand{\slirr}[1]{\Mod{L}_{#1}}              
\newcommand{\sldis}[1]{\Mod{D}_{#1}}              
\newcommand{\slindrel}[1]{\Mod{E}_{#1}}           
\newcommand{\slrel}[2]{\slindrel{#1;\Delta_{#2}}} 

\newcommand{\fock}[1]{\Mod{F}_{#1}}               
\newcommand{\latt}[2]{\Mod{V}_{#1 + \lat{#2}}}    

\newcommand{\cosetmod}[1]{\Mod{C}_{#1}}            
\newcommand{\pfmod}[2]{\cosetmod{#1;#2}}           
\newcommand{\pfirr}[2]{\pfmod{#1}{#2}^{\slirr{}}}  
\newcommand{\pfdis}[2]{\pfmod{#1}{#2}^{\sldis{}}}  
\newcommand{\pfindrel}[2]{\pfmod{#1}{#2}}
\newcommand{\pfrel}[2]{\pfindrel{#1}{#2}^{\slindrel{}}} 

\newcommand{\epfmod}[2]{\Mod{B}_{#1;#2}}             
\newcommand{\epfirr}[2]{\epfmod{#1}{#2}^{\slirr{}}}  
\newcommand{\epfdis}[2]{\epfmod{#1}{#2}^{\sldis{}}}  
\newcommand{\epfindrel}[2]{\epfmod{#1}{#2}}
\newcommand{\epfrel}[2]{\epfindrel{#1}{#2}^{\slindrel{}}} 

\newcommand{\kactable}[1]{\mathrm{Kac}(#1)} 

\newcommand{\vdim}[2]{\Delta^{\virminmod{#1}}_{(#2)}}    
\newcommand{\nsdim}[2]{\Delta^{\nslogminmod{#1}}_{(#2)}} 
\newcommand{\pfdim}[3]{\delta_{#1;#2}^{#3}}              
\newcommand{\pfirrdim}[2]{\pfdim{#1}{#2}{\slirr{}}}
\newcommand{\pfdisdim}[2]{\pfdim{#1}{#2}{\sldis{}}}
\newcommand{\pfreldim}[2]{\pfdim{#1}{#2}{\slindrel{}}}

\DeclareMathOperator{\tr}{tr}
\newcommand{\traceover}[1]{\tr_{\raisebox{-2pt}{$\scriptstyle #1$}}} 
\DeclareMathOperator{\chmap}{ch}
\newcommand{\Gr}[1]{\sqbrac[\big]{#1}}                               
\newcommand{\ch}[1]{\chmap \Gr{#1}}                                  
\newcommand{\fch}[2]{\ch{#1} \brac[\big]{#2}}                        

\newcommand{\vch}[3]{\chi^{\virminmod{#1}}_{(#2)}(#3)}               
\newcommand{\nsch}[3]{\chi^{\nslogminmod{#1}}_{(#2)}(#3)}            

\newcommand{\jth}[1]{\vartheta_{#1}}                                 
\newcommand{\fjth}[2]{\jth{#1} \brac{#2}}                            
\newcommand{\djth}[1]{\jth{#1}'}                                     
\newcommand{\fdjth}[2]{\djth{#1} \brac{#2}}                          

\newcommand{\epfirrpart}{\Gamma}
\newcommand{\epfone}[2]{\epfirrpart_{#1; #2}}                        
\newcommand{\fepfone}[3]{\epfone{#1}{#2}(#3)}

\newcommand{\modfont}[1]{\mathrm{#1}}
\newcommand{\modS}{\modfont{S}}                        

\newcommand{\Svir}[3]{\modS^{\virminmod{#1}}_{(#2)\:(#3)}}
\newcommand{\Styp}[1]{\modS^{\jth{}}_{#1}}             
\newcommand{\Stheta}[1]{\modS^{\prime}_{#1}}           
\newcommand{\Sone}[1]{\modS^{\epfirrpart}_{#1}}        

\newcommand{\vvmf}[1]{V_{#1}}                          

\newcommand{\slcat}[1]{\categ{A}_{#1}} 
\newcommand{\pfcat}[1]{\categ{C}_{#1}} 
\newcommand{\epfcat}[1]{\categ{B}_{#1}} 

\newcommand{\fus}[1]{\mathbin{\otimes_{\raisebox{-2pt}{$\scriptstyle #1$}}}}     
\newcommand{\Grfus}[1]{\mathbin{\boxtimes_{\raisebox{-2pt}{$\scriptstyle #1$}}}} 

\newcommand{\fusco}[3]{\genfrac{[}{]}{0pt}{}{#3}{#1\:#2}} 
\newcommand{\vfusco}[4]{\fusco{(#1)}{(#2)}{(#3)}_{\virminmod{#4}}}    

\DeclareMathOperator{\End}{End}

\newcommand{\lra}{\longrightarrow}

\newcommand{\dses}[3]{0 \lra #1 \lra #2 \lra #3 \lra 0} 

\newcommand{\Res}[1]{#1 \mbox{$\downarrow$} {}}           
\newcommand{\Ind}[1]{#1 \raisebox{0.18em}{$\uparrow$} {}} 

\newcommand{\cfts}{conformal field theories}

\newcommand{\lcfts}{logarithmic conformal field theories}
\newcommand{\voa}{vertex operator algebra}
\newcommand{\voas}{vertex operator algebras}
\newcommand{\svoa}{vertex operator superalgebra}
\newcommand{\svoas}{vertex operator superalgebras}
\newcommand{\vosa}{vertex operator subalgebra}
\newcommand{\vosas}{vertex operator subalgebras}
\newcommand{\wzw}{Wess-Zumino-Witten}
\newcommand{\hw}{highest-weight}
\newcommand{\hwv}{\hw{} vector}

\newcommand{\hwm}{\hw{} module}
\newcommand{\hwms}{\hw{} modules}
\newcommand{\lw}{lowest-weight}
\newcommand{\lwv}{\lw{} vector}

\newcommand{\ope}{operator product expansion}
\newcommand{\opes}{operator product expansions}
\newcommand{\lhs}{left-hand side}
\newcommand{\rhs}{right-hand side}
\newcommand{\ns}{Neveu-Schwarz}

\begin{document}

\title{Modularity of logarithmic parafermion vertex algebras}

\author{Jean Auger}
\address{
	Department of Mathematical and Statistical Sciences \\
	University of Alberta\\
	Edmonton, Canada T6G2G1
}
\email{auger1@ualberta.ca}

\author{Thomas Creutzig}
\address{
	Department of Mathematical and Statistical Sciences \\
	University of Alberta\\
	Edmonton, Canada T6G2G1
}
\email{creutzig@ualberta.ca}

\author{David Ridout}
\address{
School of Mathematics and Statistics \\
The University of Melbourne \\
Parkville, Australia, 3010}
\email{david.ridout@unimelb.edu.au}



\begin{abstract}
The parafermionic cosets $\pfvoa{k} = \Com{\hvoa}{\slvoa{k}}$ are studied for negative admissible levels $k$, as are certain infinite-order simple current extensions $\epfvoa{k}$ of $\pfvoa{k}$. Under the assumption that the tensor theory considerations of Huang, Lepowsky and Zhang apply to $\pfvoa{k}$, 
irreducible $\pfvoa{k}$- and $\epfvoa{k}$-modules are obtained from those of $\slvoa{k}$.  Assuming the validity of a certain Verlinde-type formula likewise gives 
the Grothendieck fusion rules of these irreducible modules. Notably, there are only finitely many irreducible $\epfvoa{k}$-modules.  The irreducible $\pfvoa{k}$- and $\epfvoa{k}$-characters are computed and the latter are shown, when supplemented by pseudotraces, to carry a finite-dimensional representation of the modular group. The natural conjecture then is that the $\epfvoa{k}$ are $C_2$-cofinite vertex operator algebras.
\end{abstract}

\maketitle

\onehalfspacing 

\section{Introduction}

Logarithmic conformal field theory is of significant interest in both mathematics and physics.  From its earliest appearances \cite{RozQua92,GurLog93}, it has found applications in both statistical physics and string theory whilst its characteristic feature, reducible but indecomposable modules over a \voa{}, poses significant (but rewarding) challenges to representation theorists.  As with the more familiar rational case, there are certain logarithmic theories that may be regarded as somehow archetypal \cite{CreRel11} including symplectic fermions \cite{KauSym00}, the triplet model \cite{KauExt91,GabRat96,GabLoc99}, bosonic ghosts \cite{RidBos14} and admissible-level \wzw{} models \cite{GabFus01,LesLog04,RidFus10}.

It is of fundamental importance to develop a theory of good \lcfts{}. One very natural class consists of those that are $C_2$-cofinite \cite{ZhuMod96} (or lisse), meaning that the corresponding \voa{} only possesses finitely many (isomorphism classes of) irreducible modules. It seems plausible that the representation category of a $C_2$-cofinite \voa{} is a so-called log-modular tensor category \cite{CreLog16,CreThe17}, a mild non-semisimple generalisation of the modular tensor categories that arise when restricting to the rational case \cite{HuaRig08}.  We would like to gain better intuition about log-modular tensor categories. Unfortunately, the only really well understood examples of $C_2$-cofinite \voas{} are the triplet algebras \cite{CarNon06,AdaTri08} and their close relatives \cite{AbeOrb07,AdaN=109}.  In particular, the other archetypes mentioned above are not $C_2$-cofinite, so there is an obvious need for more examples. 

The picture advocated in \cite{CreSim15,CreSch16} for constructing new examples is as follows:
\begin{equation} \label{pic:thegame}
	\begin{tikzpicture}[thick,->,scale=0.5,baseline={(al.base)}]
		\node (C) at (0,0) {\(\pfvoa{k} = \Com{\hvoa}{\savoa{k}{\alg{g}}}\)};
		\node (L) at (0,4) {\(\savoa{k}{\alg{g}}\)};
		\node (D) at (15,0) {\(\epfvoa{k} = \Com{\lvoa{L}}{\eslvoa{k}}\).};
		\node (E) at (15,4) {\(\eslvoa{k} = \bigoplus_{\sigma \in \grp{S}} \sfmod{}{\savoa{k}{\alg{g}}}\)};
		\draw (L) -- node [above] {extension} (E);
		\draw (L) -- node [left] {coset} (C);
		\draw (C) -- node [above] {extension} (D);
		\draw (E) -- node (al) [right] {coset} (D);
	\end{tikzpicture}
\end{equation}
Here, $\alg{g}$ is a simple Lie algebra, $\savoa{k}{\alg{g}}$ is the corresponding simple affine \voa{} at a negative admissible level $k$, and $\lvoa{L}$ is a lattice \voa{} extending the Heisenberg \voa{} $\hvoa$ associated to the Cartan subalgebra of $\alg{g}$. The \voa{} $\eslvoa{k}$ is a simple current extension of $\savoa{k}{\alg{g}}$ by the images of the vacuum module under a subgroup $\grp{S}$ of spectral flow functors (see \cref{sec:sl2}).  $\epfvoa{k}$ is then an extension of the parafermion \voa{} $\pfvoa{k}$ governed by an abelian intertwining algebra. 
The results of \cite{CreSch16} suggest that both $\eslvoa{k}$ and $\epfvoa{k}$ have a good chance to be $C_2$-cofinite.  However, the former has the undesirable property that its conformal weights are unbounded below, a property that is not shared by the latter if $k<0$.

Parafermionic cosets $\pfvoa{k}$ with $k \in \ZZ_{\ge 0}$ were introduced independently in the physics \cite{ZamNon85,GepNew87} and mathematics literature \cite{LepNew81,DonGen93}.  Despite the fact that physicists have long regarded these parafermions as textbook examples of rational \cfts{}, the nature of the underlying vertex algebras ($C_2$-cofinite and rational) has only recently been rigorously established \cite{DonC2C11,AraZhu14,DonRep14}.  More recent work \cite{CreSch16, CreTen17} indicates that the traditional methods employed by physicists to analyse parafermions can also be made rigorous and that these methods should also be applicable to the logarithmic parafermions that arise at admissible levels.

We mention that the physicists' parafermion chiral algebras actually correspond to generalised \voas{} obtained as finite-order simple current extensions of the $\pfvoa{k}$ cosets with $k \in \ZZ_{\ge 1}$.  In particular, the generating fields (the parafermions themselves) for $\alg{g}=\SLA{sl}{2}$ are Virasoro primaries of conformal weight $n (1-\frac{n}{k})$, where $n=1,\dots,k-1$.  When $k=1$, the parafermion theory is trivial.  When $k=2$, it follows that the parafermion chiral algebra is the free fermion, while $\pfvoa{2}$ is the Ising model \voa{}.  (This generalisation from the free fermion is in fact the reason for the name ``parafermions''.)  There are, of course, other finite-order simple current extensions of $\pfvoa{k}$, $k \in \ZZ_{\ge 1}$, that are $C_2$-cofinite vertex operator (super)algebras.  However, these are rational and therefore have a very different flavour to the admissible-level infinite-order extensions $\epfvoa{k}$ that we study here.

For $k$ admissible, the simplest case is, of course, $\alg{g}=\SLA{sl}{2}$ for which $\slvoa{k}$ is fairly well understood 
\cite{AdaVer95,CreMod12,CreMod13,RidRel15}. Our objective is to better understand the extended parafermions $\epfvoa{k}$ in the case where $k$ is also negative.
In this article, we will determine (conjecturally all) the irreducible $\epfvoa{k}$-modules, as well as a few 
reducible but indecomposable ones, and establish the modular properties of their characters. The results are consistent with the $\epfvoa{k}$ being $C_2$-cofinite, but non-rational, \voas{}. In subsequent work, we plan to prove this $C_2$-cofiniteness and study further properties of these \voas{} with the motivation being to gain a better understanding of $C_2$-cofinite \voas{} in general.

We mention two natural extensions of our study.  First, one would like to understand the parafermions of $\slvoa{-2+1/n}$ for positive integral $n$. These are non-admissible levels, but the affine \voas{} allow for large extensions that relate to the triplet algebras via quantum hamiltonian reduction \cite{CreWalg17}. We suspect that their parafermionic cosets also allow for interesting and possibly $C_2$-cofinite \voa{} extensions. Second, one has the closely related parafermionic cosets of $\savoa{k}{\SLSA{osp}{1}{2}}$ at admissible levels, also known as graded parafermions \cite{CamGra98,ForCha07}. For general admissible levels, $\savoa{k}{\SLSA{osp}{1}{2}}$ is currently under investigation \cite{RidOSP17,CreAdm17}.  The more familiar positive integer cases, together with their parafermionic cosets, are addressed in \cite{CreOSP17}, see also \cite{EnnOSP97}.

\subsection{Schur-Weyl duality for $\pfvoa{k}$} \label{sec:Schur}

Here, we summarise the relevant results of \cite{CreSch16} concerning the Heisenberg coset $\pfvoa{k} = \Com{\hvoa}{\savoa{k}{\alg{g}}}$, where $\alg{g}$ is a simple Lie algebra. 
We are mainly interested in how the decomposition of $\savoa{k}{\alg{g}}$ into $\hvoa\otimes \pfvoa{k}$-modules allows us to determine the structures of the $\pfvoa{k}$-modules and, consequently, those of the $\epfvoa{k}$-modules.

We note the following assumption that we consider to be in force throughout this paper.
\begin{assum}\label{ass:cat}
The vertex tensor category theory of Huang, Lepowsky and Zhang \cite{HuaLog10} may be applied to the $\pfvoa{k}$-module categories that we study below.
\end{assum}
\noindent We remark that the validity of \cite{HuaLog10} has been verified for the category of ordinary modules of $\slvoa{k}$ at admissible level $k$ in \cite{CreBra17} and for the Heisenberg \voa{} $\hvoa$ in \cite{CreSch16}. Determining the applicability of \cite{HuaLog10} beyond $C_2$-cofinite \voas{} is in general a very important open problem. The key conditions to prove are the closure of a given subcategory under the $P(z)$-tensor product and the convergence of products and iterates of intertwining operators. Work on the latter condition for $\pfvoa{k}$ is underway.

We now recall how the results of \cite{CreSch16} apply to our setting.
\begin{result}\label{res:VOA}
As an $\hvoa\otimes \pfvoa{k}$-module, the simple affine \voa{} $\savoa{k}{\alg{g}}$ decomposes as
\begin{equation} \label{eq:DecompAlg}
\Res{\savoa{k}{\alg{g}}} = \bigoplus_{\lambda\in \rlat} \fock{\lambda}\otimes \cosetmod{\lambda},
\end{equation}
where $\rlat$ denotes the root lattice of $\alg{g}$. Here, $\fock{\lambda}$ denotes the Fock space of $\hvoa$ with highest weight $\lambda$ and the $\cosetmod{\lambda}$ are irreducible $\pfvoa{k}$-modules.
\end{result}

\begin{result}\label{res:sc}
$\cosetmod{0}$ is the vacuum module of the simple \voa{} $\pfvoa{k}$ and the $\cosetmod{\lambda}$ are simple currents whose fusion rules include
\begin{equation}
\cosetmod{\lambda} \fus{\pfvoa{k}} \cosetmod{\mu} \cong \cosetmod{\lambda+\mu}.
\end{equation}
\end{result}
\noindent Throughout, we understand that a \emph{fusion rule} refers to the original definition used by physicists, namely the decomposition of the fusion product of two modules into isomorphism classes of indecomposables.  The multiplicities with which these indecomposables appear in a given fusion product are called the \emph{fusion multiplicities}.

\begin{result} \label{res:multfree}
For $k<0$, the decomposition \eqref{eq:DecompAlg} is \emph{multiplicity-free}, meaning that $\cosetmod{\lambda} \ncong \cosetmod{\mu}$ whenever $\lambda\neq\mu$.
\end{result}
\noindent This last result follows \cite[Sec.~3.2.1]{CreSch16} from the fact that the conformal weights of $\savoa{k}{\alg{g}}$ are bounded below.  In fact, it is also true for $\savoa{k}{\SLA{sl}{2}}$ with $k>0$ and $k \notin \ZZ$, by the criterion of \cite[Sec.~3.2.2]{CreSch16}.

\begin{result}\label{res:modules}
If $\Mod{M}$ is an indecomposable $\savoa{k}{\alg{g}}$-module on which $\hvoa$ acts semisimply, then its 
restriction to 
an $\hvoa \otimes \pfvoa{k}$-module is
\begin{equation}
\Res{\Mod{M}} \cong \bigoplus_{\mu\in \alpha + \rlat} \fock{\mu} \otimes \Mod{D}_{\mu},
\end{equation}
for some $\alpha \in \CC \otimes_{\ZZ} \rlat$.  Moreover, this decomposition is structure-preserving:  If $\Mod{M}$ has socle series $0 \subset \Mod{M}^1 \subset \dots \subset \Mod{M}^{\ell-1} \subset \Mod{M}$ and we define $\pfvoa{k}$-modules $\Mod{D}_{\mu}^i$ by
\begin{equation}
	\Res{\Mod{M}^i} \cong \bigoplus_{\mu\in \alpha + \rlat} \fock{\mu} \otimes \Mod{D}_{\mu}^i, \quad i=1,\dots,\ell-1,
\end{equation}
then $0 \subset \Mod{D}_{\mu}^1 \subset \dots \subset \Mod{D}_{\mu}^{\ell-1} \subset \Mod{D}_{\mu}$ is the socle series of $\Mod{D}_{\mu}$, for all $\mu \in \alpha + \rlat$.
\end{result}

\begin{result}\label{res:lifting}
Given an indecomposable $\pfvoa{k}$-module $\Mod{D}$, there exists $\alpha \in \CC \otimes_{\ZZ} \rlat$ such that the induction
\begin{equation}
\Ind{(\fock{\mu} \otimes \Mod{D})} = \savoa{k}{\alg{g}} \fus{\hvoa \otimes \pfvoa{k}} (\fock{\mu} \otimes \Mod{D})
\end{equation}
is an (untwisted) $\savoa{k}{\alg{g}}$-module if and only if $\mu \in \alpha + \rlat'$.
\end{result}
\noindent Here, $\rlat'$ denotes the dual lattice of $\rlat$ with respect to the bilinear form defined by the \opes{} of the generating fields of $\hvoa$.  This form is $k$ times the Killing form of $\alg{g}$, where we normalise the latter so that the length squared of the highest root is $2$.

We mention that when $\fock{\mu} \otimes \Mod{D}$ does lift to an $\savoa{k}{\alg{g}}$-module, as in the previous \lcnamecref{res:lifting}, the lift will be irreducible if and only if $\Mod{D}$ is irreducible (by \cref{res:modules}).

Given a lattice $\lat{L} \subseteq \CC \otimes_{\ZZ} \rlat$, there is a simple current extension $\lvoa{\lat{L}}$ of $\hvoa$ satisfying
\begin{equation}
	\Res{\lvoa{\lat{L}}} \cong \bigoplus_{\lambda \in \lat{L}} \fock{\lambda}.
\end{equation}
This extension is a \voa{} (superalgebra) if and only if $\lvoa{\lat{L}}$ is $\ZZ$-graded ($\frac{1}{2} \ZZ$-graded) by conformal weight \cite{DonGen93}.  A theorem of Li \cite{LiAbe01} now implies that $\pfvoa{k}$ has an extension $\epfvoa{k}$ satisfying
\begin{equation}
	\Res{\epfvoa{k}} \cong \bigoplus_{\lambda \in \lat{L}} \cosetmod{\lambda}
\end{equation}
as a $\pfvoa{k}$-module.  Moreover, $\epfvoa{k}$ is a \voa{} (superalgebra) if and only if it is $\ZZ$-graded ($\frac{1}{2} \ZZ$-graded) by conformal weight.

We now ask whether a given $\pfvoa{k}$-module $\Mod{D}$ lifts to an untwisted $\epfvoa{k}$-module
\begin{equation}
\Ind{\Mod{D}} = \epfvoa{k} \fus{\pfvoa{k}} \Mod{D}.
\end{equation}
Assuming that $\epfvoa{k}$ is $\ZZ$-graded for simplicity, the answer is that it lifts if and only if the monodromy of $\cosetmod{\lambda}$ and $\Mod{D}$ is trivial for all $\lambda \in \lat{L}$ \cite{CreSim15,KirqAn02,HuaBra15}.  When these monodromies are scalars, which happens for example when $\End \Mod{D} \cong \CC$, this triviality is decided by conformal weight considerations alone.  We record this simple conclusion for later use.

\begin{result} \label{res:sclift}
If $\epfvoa{k}$ is $\ZZ$-graded by conformal weight, then the lift $\Ind{\Mod{D}}$ of a $\pfvoa{k}$-module $\Mod{D}$ with $\End \Mod{D} \cong \CC$ is an (untwisted) $\epfvoa{k}$-module if and only if it is $\ZZ$-graded.  Moreover, if $\Mod{D}$ is irreducible as a $\pfvoa{k}$-module and it lifts, then $\Ind{\Mod{D}}$ is irreducible as a $\epfvoa{k}$-module.
\end{result}
\noindent We remark that the second statement above is quite general, but must be interpreted with care when working with simple current extensions of \svoas{}.  Then, the statement fails if one is working over $\ZZ_2$-graded modules (as one customarily does in this situation).  Specific counterexamples may be found by considering the $N=1$ superconformal minimal model of central charge $1$ whose order $2$ simple current extension is the $N=2$ superconformal minimal model of level $1$.

Finally, to compute the (Grothendieck) fusion rules of $\pfvoa{k}$ and $\epfvoa{k}$, we also need the fact that the induction functor $\Ind{}$ is a tensor functor of vertex tensor categories. This is the main theorem of \cite{CreTen17} and was previously conjectured for simple current extensions in \cite{RidVer14}. For this computation, we need a further assumption. 

For $k$ admissible, there is a category $\slcat{k}$ of finite-length $\slvoa{k}$-modules that we define explicitly below after \cref{relevantsimples}.  If $\slcat{k}$ is closed under fusion products and fusing with any module defines an exact endofunctor on it, then fusion descends to a product on the Grothendieck group of $\slcat{k}$ (in which a module is identified with the sum of its composition factors).  We shall refer to the decomposition of a given Grothendieck fusion product into images of irreducibles as a \emph{Grothendieck fusion rule} and to the multiplicities of said irreducibles as \emph{Grothendieck fusion coefficients}.
\begin{assum} \label{ass:fusion}
For $k$ admissible, the Grothendieck fusion rules of $\slcat{k}$ are well defined and the Grothendieck fusion coefficients are computed by the \emph{standard Verlinde formula} of \cite{CreLog13,RidVer14}.
\end{assum}
\noindent With this highly non-trivial assumption, the Grothendieck fusion coefficients of $\slvoa{k}$ have been computed in 
\cite{CreMod12,CreMod13}, 
see also \cref{sec:sl2fus}.  The results are consistent with the fusion rules that have been computed \cite{GabFus01,RidFus10} for $k=-\frac{4}{3}$ and $-\frac{1}{2}$.  We remark that this assumption is actually a theorem for the category of ordinary $\slvoa{k}$-modules, for $k$ admissible, by \cite[Cor.~7.7]{CreBra17}.

With this assumption, we can compute the Grothendieck fusion rules for $\pfvoa{k}$ and $\epfvoa{k}$ from those of $\slvoa{k}$ using the following result.
\begin{result}\label{res:fus}
For any $\pfvoa{k}$-modules $\Mod{D}$ and $\Mod{E}$, choose $\delta, \varepsilon \in \CC \otimes_{\ZZ} \rlat$ such that $\Ind{\left(\fock{\delta}\otimes \Mod{D}\right)}$ and $\Ind{\left(\fock{\varepsilon}\otimes \Mod{E}\right)}$
are $\savoa{k}{\alg{g}}$-modules. Then,
\begin{equation}
\Ind{(\fock{\delta}\otimes \Mod{D})} \fus{\savoa{k}{\alg{g}}} \Ind{(\fock{\varepsilon}\otimes \Mod{E})}
\cong \Ind{\brac[\big]{\fock{\delta+\varepsilon}\otimes (\Mod{D} \fus{\pfvoa{k}} \Mod{E})}}.
\end{equation}
Moreover, if both $\Ind{\Mod{D}}$ and $\Ind{\Mod{E}}$ are $\epfvoa{k}$-modules, then
\begin{equation}
\Ind{\Mod{D}}  \fus{\epfvoa{k}} \Ind{\Mod{E}} \cong  \Ind{(\Mod{D}  \fus{\pfvoa{k}} \Mod{E})}.
\end{equation}
\end{result}

\subsection{Outline and results}

We begin in \cref{sec:sl2} with a detailed review of the admissible-level simple \voas{} $\slvoa{k}$, their representation theories (category $\categ{O}$ and beyond), and their Grothendieck fusion rules.  This serves to fix notation and conventions as well as collect the various results that will be needed for what follows.  We adopt here the language of the \emph{standard module formalism} (standard, typical and atypical) introduced for \lcfts{} in \cite{CreLog13,RidVer14}.  As it is also convenient for describing the representation theory of the parafermion and extended parafermion coset \voas{}, we use this language throughout.

The parafermion cosets $\pfvoa{k}$, with $k$ admissible and negative, are analysed in \cref{sec:pf}.  We first determine the decompositions of the standard and irreducible atypical $\slvoa{k}$-modules as $\hvoa \otimes \pfvoa{k}$-modules, thereby deducing the characters of the standard and irreducible atypical $\pfvoa{k}$-modules (\cref{cosetEtype,charLcoset}).  The Grothendieck fusion rules of the irreducible $\pfvoa{k}$-modules are then presented (\cref{prop:pffusL,prop:pfgrfus}).  We conclude this section with a brief discussion of three ``small'' examples.  After recalling that $\pfvoa{-1/2}$ and $\pfvoa{-4/3}$ have been previously identified \cite{AdaCon05,RidSL210} with the singlet \voas{} $\singvoa{1,2}$ and $\singvoa{1,3}$, respectively, we concentrate on $\pfvoa{-2/3}$ and find that it is isomorphic to the bosonic orbifold of the $N=1$ supersinglet \svoa{} $\ssingvoa{1,3}$ introduced by Adamovi\'{c} and Milas \cite{AdaN=109}.

\cref{sec:epf} then addresses the extended parafermion cosets $\epfvoa{k}$, again assuming that $k$ is admissible and negative.  Identifying $\epfvoa{k}$ as an infinite-order simple current extension of $\pfvoa{k}$, we show that the former has only finitely many irreducible modules (\cref{epffinite}).  After reporting on the Grothendieck fusion rules (\cref{prop:epffusL,prop:epfgrfus}) and noting that $\epfvoa{-1/2}$, $\epfvoa{-4/3}$ and $\epfvoa{-2/3}$ coincide with certain explicitly described orbifolds of the (super)triplets $\tripvoa{1,2}$, $\tripvoa{1,3}$ and $\stripvoa{1,3}$, respectively, we compute the irreducible $\epfvoa{k}$-characters (\cref{prop:epfstchars,prop:epfirrchar=theta,rem:D=E+L}).

We conclude with a detailed investigation of the modular properties of the $\epfvoa{k}$-characters.  Those of the standard modules are easily deduced (\cref{prop:epfstcharmod}), whilst those of the atypical irreducibles are much more subtle.  Our main result here is \cref{thm:atypicalmodularity} which states that the parts of the atypical irreducible characters of modular weight $1$ define a finite-dimensional vector-valued modular form.  It also gives an explicit upper bound for the dimension that we believe is sharp.  As the remaining parts may be expressed as linear combinations of standard characters, it follows that the irreducible $\epfvoa{k}$-characters and the objects 
obtained by multiplying the weight $1$ parts of the atypical irreducible characters by the modular parameter $\tau$, together span a finite-dimensional representation of the modular group.  We expect that the latter objects correspond to pseudotraces.  As this is precisely how the irreducible characters of a $C_2$-cofinite \voa{} \cite{MiyMod04} behave under modular transformations, we conjecture that the $\epfvoa{k}$ are $C_2$-cofinite for all $k$ admissible and negative (noting that this is already known for $k=-\frac{1}{2}$, $-\frac{4}{3}$ and $-\frac{2}{3}$).  We intend to prove this $C_2$-cofiniteness in a sequel.

\section*{Acknowledgements}

We thank Shashank Kanade and Andrew Linshaw for discussions relating to the results presented here.
J.~A. is supported by a Doctoral Research Scholarship from the Fonds de Recherche Nature et Technologies de Qu\'ebec (184131).  T.~C. is supported by the Natural Sciences and Engineering Research Council of Canada (RES0020460).
DR's research is supported by the Australian Research Council Discovery Project DP160101520 and the Australian Research Council Centre of Excellence for Mathematical and Statistical Frontiers CE140100049.

\section{The simple \voa{} $\slvoa{k}$} \label{sec:sl2}

In this section, we review those aspects of the admissible-level \voas{} $\slvoa{k}$ that will be required for our analysis of the corresponding parafermion cosets.  Our main sources for this are \cite{AdaVer95,CreMod13,RidRel15} to which we refer for additional details and references.

\subsection{Representation theory} \label{sec:sl2reps}

Recall the standard basis $\set{e,h,f}$ of $\SLA{sl}{2}$ in which the Cartan subalgebra $\alg{h}$ is spanned by $h$.  The non-zero commutators are specified by
\begin{equation}
	\comm{h}{e} = 2e, \qquad \comm{e}{f} = h, \qquad \comm{h}{f} = -2f
\end{equation}
and the Killing form is normalised so that
\begin{equation}
	\inner{h}{h} = 2, \qquad \inner{e}{f} = 1.
\end{equation}
The affine Kac-Moody algebra $\AKMA{sl}{2}$ is then defined to be the universal central extension of the loop algebra $\SLA{sl}{2} \otimes \CC[x,x^{-1}]$.  We choose generators $e_n$, $h_n$ and $f_n$ in $\AKMA{sl}{2}$, for $n \in \ZZ$, such that they project onto $e \otimes x^n$, $h \otimes x^n$ and $f \otimes x^n$, respectively, upon quotienting by the centre.  The central element is denoted by $K$ and its eigenvalue is the level $k$.

Fix $k \in \CC$ and consider the Verma module of $\AKMA{sl}{2}$ whose highest weight is $(k-\lambda) \omega_0 + \lambda \omega_1$, for some $\lambda \in \CC$, where $\omega_0$ and $\omega_1$ are the fundamental weights.  Its irreducible quotient will be denoted by $\sldis{\lambda}^+$.  When $\lambda \in \NN$, we shall also denote this irreducible quotient by $\slirr{\lambda+1}$ to emphasise that its ground states (the states of minimal conformal weight) form a finite-dimensional subspace of dimension $\lambda+1$.  The irreducible $\AKMA{sl}{2}$-module $\sldis{0}^+ = \slirr{1}$ is well known to carry the structure of a simple \voa{}, which we denote by $\slvoa{k}$, provided that $k \neq -2$.  We exclude this \emph{critical} level from considerations.

We recall certain properties of the simple \voa{} $\slvoa{k}$ and its modules, when the level $k$ is \emph{admissible}:
\begin{equation}
	k+2 = t = \frac{u}{v}, \qquad u \in \ZZ_{\ge 2}, \quad v \in \ZZ_{\ge 1}, \quad \gcd \set{u,v} = 1.
\end{equation}
We shall also define, for later convenience, $w = -kv = 2v-u$.  The level $k$ being admissible is equivalent to the level $k$ universal \voa{} associated to $\AKMA{sl}{2}$ being non-simple.  The central charge of $\slvoa{k}$ is $c = 3 - \frac{6}{t}$.

\begin{thm}[Adamovi\'c-Milas \cite{AdaVer95} and Dong-Li-Mason \cite{DonVer97}] \label{hwsimples}
	Let $k = -2 + \frac{u}{v}$ be an admissible level and let
	\begin{equation} \label{eq:DefLambda}
		\lambda_{r,s} = r - 1 - ts.
	\end{equation}
	Then, the \hw{} $\slvoa{k}$-modules are exhausted, up to isomorphism, by those of the $\slirr{r} = \slirr{\lambda_{r,0}+1}$, for $r = 1,\dots,u-1$, and the $\sldis{r,s}^+ = \sldis{\lambda_{r,s}}^+$, for $r = 1,\dots,u-1$ and $s = 1,\dots,v-1$.
\end{thm}

\begin{remark}
	The weights $(k - \lambda_{r,s}) \omega_0 + \lambda_{r,s} \omega_1$, for $r = 1,\dots,u-1$ and $s = 0,\dots,v-1$, are the highest weights of the admissible level-$k$ $\AKMA{sl}{2}$-modules, as defined by Kac and Wakimoto \cite{KacMod88}. This original definition of admissibility was motivated by the observation that these irreducible modules admit a generalisation of the Weyl-Kac character formula for integrable modules.
\end{remark}

The conformal weight of the ground states of each of these \hw{} $\slvoa{k}$-modules is determined by its highest weight, hence by the parameters $r$ and $s$.  Specifically, the conformal weight of each ground state of $\sldis{r,s}^+$ is given by
\begin{equation} \label{eq:DefDelta}
	\Delta_{r,s} = \frac{(r-ts)^2-1}{4t} = \frac{(vr-us)^2-v^2}{4uv},
\end{equation}
where $r = 1,\dots,u-1$ and $s = 0,\dots,v-1$.  We note the symmetries
\begin{equation} \label{eq:KacSymm}
	\lambda_{u-r,v-s} = -\lambda_{r,s} - 2, \qquad \Delta_{u-r,v-s} = \Delta_{r,s}.
\end{equation}

It turns out that for non-integer admissible levels, these being those with $v>1$, it is not sufficient to consider only \hwms{}.  Instead, one is forced \cite{GabFus01,RidFus10} to broaden the class of modules to include, in particular, the \emph{relaxed} \hwms{}.  These are defined as modules that are generated by a relaxed \hwv{}, this being a weight vector that is annihilated by every mode with positive index.  For $\AKMA{sl}{2}$, this means an eigenstate of $h_0$ (of definite level) which is annihilated by the positive modes $e_n$, $h_n$ and $f_n$, $n > 0$.  A normal \hwv{} for $\AKMA{sl}{2}$ is therefore a relaxed \hwv{} that happens to be also annihilated by $e_0$.

Before presenting the classification of irreducible relaxed \hw{} $\slvoa{k}$-modules, we recall that the Cartan subalgebra-preserving automorphisms of $\AKMA{sl}{2}$ define an infinite group, isomorphic to $\ZZ_2 \ltimes \ZZ$, of invertible functors acting on $\slvoa{k}$-modules.  This action is called \emph{twisting} by the automorphism.  This group is isomorphic to the affine Weyl group, but the free part should actually be identified with translations by the dual of the root lattice $\rlat$ rather than the coroot lattice \cite{RidSL208}.  The torsion part of this group has a generator $\conjaut$ called \emph{conjugation} that may be identified with the generator of the Weyl group of $\SLA{sl}{2}$.  Twisting by conjugation therefore negates $\SLA{sl}{2}$-weights but leaves conformal weights (and the level) unchanged.

Consider now an irreducible weight module over $\SLA{sl}{2}$ that is neither highest- nor \lw{}.  Identifying $\SLA{sl}{2}$ with the horizontal subalgebra of $\AKMA{sl}{2}$, we extend this to a module over the subalgebra generated by the modes of non-negative index by requiring that its level is $k$ and that the positive modes of $\AKMA{sl}{2}$ act trivially.  Inducing to an $\AKMA{sl}{2}$-module now results in a relaxed \hwm{} that is determined by the class $\lambda \omega_1 + \rlat$ of its $\SLA{sl}{2}$-weights, modulo $\rlat$, and the conformal weight $\Delta$ of its ground states, which is fixed by the eigenvalue of the quadratic Casimir on the original irreducible weight module.

We shall identify the root lattice $\rlat$ of $\SLA{sl}{2}$ with $2\ZZ$ throughout and will frequently abuse notation by identifying $\lambda \in \CC$ with $\lambda \omega_1 \in \alg{h}^*$.  It follows that the level $k$ relaxed \hwm{} just constructed is parametrised by $\lambda + \rlat \in \alg{h}^* / \rlat = \CC / 2\ZZ$ and $\Delta \in \CC$.  Let $\slrel{\lambda}{}$ denote the unique irreducible quotient of this relaxed \hwm{} so that $\slrel{\lambda}{} = \slrel{\mu}{}$ if $\lambda = \mu \pmod{\rlat}$.  It is likewise a relaxed \hwm{}.

\begin{thm}[Adamovi\'c-Milas \cite{AdaVer95}, see also \cite{RidRel15}] \label{rhwsimples}
	Let $k = -2 + \frac{u}{v}$ be an admissible level. Then, the irreducible relaxed \hw{} $\slvoa{k}$-modules are exhausted, up to isomorphism, by the following list:
	\begin{itemize}
		\item The $\slirr{r}$, for $r = 1,\dots,u-1$;
		\item The $\sldis{r,s}^+$, for $r = 1,\dots,u-1$ and $s = 1,\dots,v-1$;
		\item The conjugates $\sldis{r,s}^- = \conjmod{\sldis{r,s}^+}$, for $r = 1,\dots,u-1$ and $s = 1,\dots,v-1$;
		\item The $\slrel{\lambda}{r,s}$, for $r = 1,\dots,u-1$, $s = 1,\dots,v-1$ and $\lambda \in \alg{h}^*$ with $\lambda \neq \lambda_{r,s}, \lambda_{u-r,v-s} \pmod{\rlat}$.
		\end{itemize}
	Apart from the identifications $\slrel{\lambda}{r,s} = \slrel{\lambda}{u-r,v-s}$, that follow trivially from \eqref{eq:KacSymm}, and $\slrel{\lambda}{r,s} = \slrel{\mu}{r,s}$, if $\lambda = \mu \pmod{\rlat}$, the modules in this list are all mutually non-isomorphic.
\end{thm}
\begin{remark}
	The caveat that $\lambda \neq \lambda_{r,s}, \lambda_{u-r,v-s} \pmod{\rlat}$ in the classification of the irreducible relaxed \hw{} $\slvoa{k}$-modules arises from the fact that an $\slindrel{}$-type module with parameters $(\lambda_{r,s}; \Delta_{r,s})$ or $(\lambda_{u-r,v-s}; \Delta_{r,s})$ must be reducible.  Indeed, the irreducible weight $\SLA{sl}{2}$-module from which we induced must have either a highest- or \lwv{}.  Of course, one may also induce from reducible weight $\SLA{sl}{2}$-modules.  For each $r = 1,\dots,u-1$ and $s = 1,\dots,v-1$, one thereby arrives at two distinct relaxed \hw{} $\slvoa{k}$-modules, both reducible but indecomposable, by quotienting the induced module by the (unique) maximal proper submodule whose intersection with the space of ground states is zero.  We denote these $\slvoa{k}$-modules by $\slindrel{r,s}^+ = \slrel{\lambda_{r,s}}{r,s}^+$ and $\slindrel{r,s}^- = \slrel{\lambda_{r,s}}{r,s}^-$, noting that they are characterised, up to isomorphism, by the following non-split short exact sequences:
	\begin{equation} \label{es:DED}
		\dses{\sldis{r,s}^+}{\slindrel{r,s}^+}{\sldis{u-r,v-s}^-}, \qquad
		\dses{\sldis{r,s}^-}{\slindrel{r,s}^-}{\sldis{u-r,v-s}^+}.
	\end{equation}
	These exact sequences were originally stated in \cite{RidRel15}.  A rigorous justification will appear in \cite{KawRel18}.
\end{remark}
\begin{remark} \label{IdentConjMod}
	We observe that the $\slirr{r}$, $r=1,\dots,u-1$, are self-conjugate: $\conjmod{\slirr{r}} \cong \slirr{r}$.  The conjugates of the relaxed \hw{} $\slvoa{k}$-modules are given by $\conjmod{\slrel{\lambda}{r,s}} \cong \slrel{-\lambda}{r,s}$ and $\conjmod{\slindrel{r,s}^{\pm}} \cong \slindrel{r,s}^{\mp}$.
\end{remark}

While \cref{rhwsimples} classifies the irreducible $\slvoa{k}$-modules in the category of relaxed \hw{} $\AKMA{sl}{2}$-modules, it is still easy to construct irreducible $\slvoa{k}$-modules that are not isomorphic to those introduced so far.  This construction uses the free part of the automorphism group $\ZZ_2 \ltimes \ZZ$, the elements of which are called \emph{spectral flow} automorphisms.  We choose the generator $\sfaut$ of the free part as in \cite{CreMod13,RidRel15} so that
$\conjaut \sfaut = \sfaut^{-1} \conjaut$ and the following isomorphisms hold.
\begin{prop}[\cite{CreMod13}] \label{twistRules}
Fix an admissible level $k = -2 + \frac{u}{v}$ and assume that $v > 1$. Then, we have
\begin{equation} \label{eq:SFIdentifications}
	\begin{gathered}
		\sfmod{}{\slirr{r}} \cong \sldis{u-r,v-1}^+,\quad \sfmod{-1}{\slirr{r}} \cong \sldis{u-r,v-1}^-, \\
		\sfmod{-1}{\sldis{r,s}^+} \cong \sldis{u-r,v-1-s}^-,
	\end{gathered}
	\qquad
	\begin{gathered}
		r=1,\dots,u-1, \\
		r=1,\dots,u-1,\quad s=1,\dots,v-2. \\
	\end{gathered}
\end{equation}
Together with the isomorphisms of \cref{IdentConjMod}, these generate a complete set of isomorphisms among twists of the $\slvoa{k}$-modules introduced above.
\end{prop}

\begin{remark}
	When $v=1$, so that $k$ is a non-negative integer, the list of irreducible relaxed \hw{} $\slvoa{k}$-modules in \cref{rhwsimples} collapses to just the $\slirr{r}$, with $r=1,\dots,u-1=k+1$.  The corresponding \cfts{} are the rational \wzw{} models describing strings on $\SLG{SU}{2}$ \cite{WitNon84}.  In this case, the isomorphisms involving spectral flow are generated by $\sfmod{}{\slirr{r}} \cong \slirr{u-r}$.
\end{remark}

Because twisting by an automorphism is an invertible functor, it preserves structure.  In particular, twisting an irreducible module results in another irreducible module.  Moreover, since these automorphisms lift to automorphisms of the affine \emph{vertex algebra}, it is easy to see that each twist of an $\slvoa{k}$-module results in another $\slvoa{k}$-module.  Spectral flow therefore gives us an infinite set of new irreducible $\slvoa{k}$-modules for each of the irreducible $\slvoa{k}$-modules listed in \cref{rhwsimples}.

\begin{remark} \label{relevantsimples}
Fix an admissible level $k = -2 + \frac{u}{v}$ and assume that $v > 1$. Then, the irreducible $\slvoa{k}$-modules of interest in this paper consist of the following:
\begin{itemize}
\item The $\sfmod{\ell}{\sldis{r,s}^+}$, with $\ell \in \ZZ$, $r = 1,\dots,u-1$ and $s = 1,\dots,v-1$;
\item The $\sfmod{\ell}{\slrel{\lambda}{r,s}}$, with $\ell \in \ZZ$, $r = 1,\dots,u-1$, $s = 1,\dots,v-1$ and $\lambda \in (\RR/2\ZZ) \setminus \set{\lambda_{r,s}, \lambda_{u-r,v-s}}$.
\end{itemize}
Again, aside from $\sfmod{\ell}{\slrel{\lambda}{r,s}} = \sfmod{\ell}{\slrel{\lambda}{u-r,v-s}} = \sfmod{\ell}{\slrel{\mu}{r,s}}$, where $\lambda = \mu \pmod{\rlat}$, the modules in this list are all mutually non-isomorphic.  The restriction to real $\SLA{sl}{2}$ weights in $\RR / 2 \ZZ = \alg{h}^*_{\RR} / \rlat$ is physically motivated and confirmed by modular considerations \cite{CreMod12,CreMod13}.
\end{remark}

We shall suppose throughout that we are working in the full subcategory $\slcat{k}$ of $\slvoa{k}$-modules in which the simple objects are the irreducibles of \cref{relevantsimples} and every object is a subquotient of the (iterated) fusion product of a finite collection of simple objects.  This, of course, presupposes that these fusion products are sufficiently nice (for instance, we believe that they all have finite composition length) and that each of their irreducible subquotients is isomorphic to one of the irreducibles introduced above.  For the admissible levels $k=-\frac{4}{3}$ and $-\frac{1}{2}$, the explicit fusion calculations reported 
in \cite{GabFus01,RidFus10} lead us to expect 
that this niceness continues to hold for all admissible levels (see \cref{sec:sl2fus} for some relevant results in this direction).

As discussed in \cite{CreSim15}, the category $\slcat{k}$ is interesting because it is expected to contain the projective covers of the irreducible $\slvoa{k}$-modules.  Whilst rigorously identifying these projective covers remains out of reach here (as is the case for almost all logarithmic \voas{}), 
the corresponding subcategory for the $\mathcal{W}(p)$-triplet algebras does contain all of the indecomposable projectives, according to \cite{NagTri11,TsuTen12}.

\begin{remark} \label{rem:whittaker}
	For $k$ admissible and non-integral, there are irreducible $\slvoa{k}$-modules besides those listed in \cref{rhwsimples} and their spectral flows.  In particular, there are the Whittaker-type modules constructed in \cite{AdaRea17}.  However, they seem to play no role in the modular properties of \cite{CreMod13} and, being non-weight, they cannot appear in fusion products of the irreducibles of \cref{rhwsimples} \cite{GabFus94}.  We shall therefore ignore these non-weight modules in what follows.
\end{remark}

\subsection{Characters and modularity} \label{sec:sl2chars}

The character of an $\AKMA{sl}{2}$-module $\Mod{M}$ is defined to be the following formal power series in $y$, $z$ and $q$:
\begin{equation}
	\fch{\Mod{M}}{y,z,q} = \traceover{\Mod{M}} y^k z^{h_0} q^{L_0-c/24}.
\end{equation}
At the level of characters, twisting an $\AKMA{sl}{2}$-module $\Mod{M}$ by conjugation or spectral flow yields
\begin{equation} \label{eq:TwChar}
	\fch{\conjmod{\Mod{M}}}{y,z,q} = \fch{\Mod{M}}{y,z^{-1},q}, \qquad
	\fch{\sfmod{\ell}{\Mod{M}}}{y,z,q} = \fch{\Mod{M}}{y z^{\ell} q^{\ell^2 / 4},z q^{\ell / 2}, q},
\end{equation}
assuming, of course, that $\Mod{M}$ has finite-dimensional weight spaces (so that characters exist).

For $u,v \in \ZZ_{\ge 2}$ coprime, let $\virminmod{u,v}$ denote the Virasoro minimal model \voa{} of central charge $1 - \frac{6 (v-u)^2}{uv}$ and, for $r = 1, \dots, u-1$ and $s = 1, \dots, v-1$, let $(r,s)$ denote the irreducible $\virminmod{u,v}$-module whose ground states have conformal weight
\begin{equation}
	\vdim{u,v}{r,s} = \frac{(vr-us)^2 - (v-u)^2}{4uv}.
\end{equation}
The character of this module will be denoted by $\vch{u,v}{r,s}{q}$.  For completeness, we give an explicit formula:
\begin{equation}
	\vch{u,v}{r,s}{q} = \frac{1}{\eta(q)} \sum_{n \in \ZZ} \sqbrac*{q^{(2uvn + vr-us)^2/4uv} - q^{(2uvn + vr+us)^2/4uv}}.
\end{equation}
\begin{prop} \label{plainslslchar}
	Fix an admissible level $k = -2 + \frac{u}{v}$ and assume that $v>1$.  Then, we have the following character formulae:
	\begin{subequations} \label{eq:sl2chars}
		\begin{align}
			\ch{\sfmod{\ell}{\slrel{\lambda}{r,s}}} &= \frac{y^k z^{\ell k} q^{\ell^2 k / 4} \vch{u,v}{r,s}{q}}{\eta(q)^2} \sum_{\mu \in \lambda + \rlat} z^{\mu} q^{\ell \mu / 2}, \label{eq:typchar} \\
			\ch{\sfmod{\ell}{\slindrel{r,s}^+}} &= \frac{y^k z^{\ell k} q^{\ell^2 k / 4} \vch{u,v}{r,s}{q}}{\eta(q)^2} \sum_{\mu \in \lambda_{r,s} + \rlat} z^{\mu} q^{\ell \mu / 2}, \label{eq:atypstchar} \\
			\ch{\sfmod{\ell}{\slirr{r}}} &= \sum_{s'=1}^{v-1} (-1)^{s'-1} \sum_{m=0}^{\infty} \brac*{\ch{\sfmod{2mv+s'+\ell}{\slindrel{r,s'}^+}} - \ch{\sfmod{2(m+1)v-s'+\ell}{\slindrel{u-r,v-s'}^+}}}, \label{eq:atypLchar} \\
			\ch{\sfmod{\ell}{\sldis{r,s}^+}} &= \sum_{s'=s+1}^{v-1} (-1)^{s'-s-1} \ch{\sfmod{s'-s+\ell}{\slindrel{r,s'}^+}} + (-1)^{v-1-s} \ch{\sfmod{v-s+\ell}{\slirr{u-r}}}. \label{eq:atypDchar}
		\end{align}
	\end{subequations}
	If $k<0$, then the infinite sum in \eqref{eq:atypLchar} converges in the sense of formal power series in $z$, meaning that the coefficient of each power of $z$ converges to a meromorphic function of $q$ (for $\abs{q} < 1$).
\end{prop}

\begin{remark}
	The character formulae given in \cref{eq:atypstchar,eq:atypLchar,eq:atypDchar} were originally derived in \cite[Prop.~4 and Prop.~8]{CreMod13}, while \eqref{eq:typchar} was stated without proof.  Recently, a proof for generic values of $\lambda$, $r$ and $s$ was given in \cite{AdaRea17} using an explicit construction of the modules.  A full proof will appear in \cite{KawRel18}.
\end{remark}

\begin{remark} \label{slatypchar-}
	It is easy to check that $\ch{\sfmod{\ell}{\slindrel{r,s}^-}} = \ch{\sfmod{\ell}{\slindrel{u-r,v-s}^+}}$ using \eqref{es:DED} and the exactness of the spectral flow functor.
\end{remark}
\begin{remark}
	In the standard module formalism introduced in \cite{CreLog13,RidVer14}, the irreducibles $\sfmod{\ell}{\slrel{\lambda}{r,s}}$ are the \emph{typical} $\slvoa{k}$-modules.  The reducible but indecomposable modules $\sfmod{\ell}{\slindrel{r,s}^{\pm}}$ are examples of \emph{atypical} $\slvoa{k}$-modules.  Together, these two classes form the \emph{standard} modules of $\slvoa{k}$.  The $\sfmod{\ell}{\sldis{r,s}^+}$ are likewise atypical, but are not standard: their characters are expressible as infinite linear combinations of (atypical) standard characters.  However, these character formulae only converge as formal power series if $k<0$.  The question of what this means for $k>0$ will not be addressed here.
\end{remark}

\begin{remark} \label{rem:slres}
	The character formulae given above for the irreducible atypicals were deduced from the following resolutions \cite{CreMod13}:
	\begin{subequations} \label{eq:slres}
		\begin{align} \label{eq:slresL}
			\cdots &\lra \sfmod{3v-1}{\slindrel{r,v-1}^+} \lra \cdots \lra \sfmod{2v+2}{\slindrel{r,2}^+} \lra \sfmod{2v+1}{\slindrel{r,1}^+} \notag \\
			&\qquad\qquad \lra \sfmod{2v-1}{\slindrel{u-r,v-1}^+} \lra \cdots \lra \sfmod{v+2}{\slindrel{u-r,2}^+} \lra \sfmod{v+1}{\slindrel{u-r,1}^+} \notag \\
			&\qquad\qquad\qquad\qquad \lra \sfmod{v-1}{\slindrel{r,v-1}^+} \lra \cdots \lra \sfmod{2}{\slindrel{r,2}^+} \lra \sfmod{}{\slindrel{r,1}^+} \lra \slirr{r} \lra 0,
		\end{align}
		\begin{equation} \label{eq:slresD}
			0 \lra \sfmod{v-s}{\slirr{u-r}} \lra \sfmod{v-1-s}{\slindrel{r,v-1}^+} \lra \cdots \lra \sfmod{2}{\slindrel{r,s+2}^+} \lra \sfmod{}{\slindrel{r,s+1}^+} \lra \sldis{r,s}^+ \lra 0.
		\end{equation}
	\end{subequations}
	For $s=v-1$, the latter resolution reduces to the isomorphism $\sfmod{}{\slirr{u-r}} \cong \sldis{r,s}^+$.  One can obtain other resolutions by applying the conjugation and/or contragredient dual functors to \eqref{eq:slres} and these may lead to somewhat different looking character formulae.  For example, conjugating \eqref{eq:slresL} leads to
	\begin{align} \label{eq:atypLchar'}
		\ch{\sfmod{\ell}{\slirr{r}}} &= \sum_{s'=1}^{v-1} (-1)^{s'-1} \sum_{m=0}^{\infty} \brac*{\ch{\sfmod{-2mv-s'+\ell}{\slindrel{r,s'}^-}} - \ch{\sfmod{-2(m+1)v+s'+\ell}{\slindrel{u-r,v-s'}^-}}} \notag \\
		&= \sum_{s'=1}^{v-1} (-1)^{s'-1} \sum_{m=0}^{\infty} \brac*{\ch{\sfmod{-2mv-s'+\ell}{\slindrel{u-r,v-s'}^+}} - \ch{\sfmod{-2(m+1)v+s'+\ell}{\slindrel{r,s'}^+}}},
	\end{align}
	using \cref{IdentConjMod,slatypchar-}.
\end{remark}

\subsection{Fusion} \label{sec:sl2fus}

The fusion rules of the admissible-level $\slvoa{k}$-modules have only been (partially) computed for $k=-\frac{4}{3}$ \cite{GabFus01} and $k=-\frac{1}{2}$ \cite{RidSL208,RidFus10}.  However, the \emph{Grothendieck} fusion rules are known for all $k$ \cite{CreMod12,CreMod13}, subject to two assumptions.  The first is that fusing with an irreducible $\slvoa{k}$-module defines an exact functor on the module category $\slcat{k}$ (see \cref{sec:sl2reps}), so that the fusion product induces a ring structure on the Grothendieck group of $\slcat{k}$.  The second is that the structure constants of the Grothendieck fusion product are computed by the \emph{standard Verlinde formula} of \cite{CreLog13,RidVer14}, which we understood first in the example of $\VOA{V}_k\left(\SLSA{gl}{1}{1}\right)$ \cite{CreBra07,CreRel11}.  We shall assume that these Grothendieck fusion rules are correct, hence that the fusion rules are known up to ambiguities involving non-trivial extensions.  This was the content of \cref{ass:fusion}, as stated in \cref{sec:Schur}.

In general, we let $\fus{\VOA{V}}$ 
denote the fusion product of a given \voa{} $\VOA{V}$ and $\Grfus{\VOA{V}}$ its Grothendieck fusion product.  
The image of a $\VOA{V}$-module $\Mod{M}$ in the Grothendieck ring of $\VOA{V}$ shall be denoted by $\Gr{\Mod{M}}$.

Before stating the $\slvoa{k}$ fusion rules, it is convenient to introduce some notation for the fusion rules of the Virasoro minimal model \voa{} $\virminmod{u,v}$.  Let $\vfusco{r,s}{r',s'}{r'',s''}{u,v}$ denote the fusion coefficient involving the irreducible $\virminmod{u,v}$-modules $(r,s)$, $(r',s')$ and $(r'',s'')$, so that
\begin{equation} \label{eq:virminmodfus}
	(r,s) \fus{\virminmod{u,v}} (r',s') \cong \bigoplus_{(r'',s'')} \vfusco{r,s}{r',s'}{r'',s''}{u,v} (r'',s'').
\end{equation}
Here, the direct sum runs over the irreducible $\virminmod{u,v}$-modules $(r'',s'')$ in the Kac table
\begin{equation} \label{eq:DefKacTable}
	\kactable{u,v} = \frac{\set{1, \dots, u-1} \times \set{1, \dots, v-1}}{(r,s) \sim (u-r,v-s)}.
\end{equation}
In what follows, sums indexed by irreducible $\virminmod{u,v}$-modules will always be assumed to run over $\kactable{u,v}$.  Note that because $\virminmod{u,v}$ is rational, its fusion rules and Grothendieck fusion rules coincide.

\begin{proposition}[\protect{\cite[Props.~14, 15 and 18]{CreMod13}}] \label{prop:slfusL}
	Fix an admissible level $k = -2 + \frac{u}{v}$ and assume that $v>1$.  Given \cref{ass:fusion}, 
	the fusion rules of the irreducible $\slvoa{k}$-modules with the $\sfmod{\ell}{\slirr{r}}$ are then
	\begin{subequations}
		\begin{align}
			\sfmod{\ell}{\slirr{r}} \fus{\slvoa{k}} \sfmod{\ell'}{\slirr{r'}} &\cong \bigoplus_{r''=1}^{u-1} \vfusco{r,1}{r',1}{r'',1}{u,v} \sfmod{\ell+\ell'}{\slirr{r''}}, \label{fr:slLL} \\
			\sfmod{\ell}{\slirr{r}} \fus{\slvoa{k}} \sfmod{\ell'}{\sldis{r',s'}^+} &\cong \bigoplus_{r''=1}^{u-1} \vfusco{r,1}{r',1}{r'',1}{u,v} \sfmod{\ell+\ell'}{\sldis{r'',s'}^+}, \\
			\sfmod{\ell}{\slirr{r}} \fus{\slvoa{k}} \sfmod{\ell'}{\slrel{\lambda'}{r',s'}} &\cong \bigoplus_{r''=1}^{u-1} \vfusco{r,1}{r',1}{r'',1}{u,v} \sfmod{\ell+\ell'}{\slrel{r-1 + \lambda'}{r'',s'}}.
					\end{align}
	\end{subequations}
	When $v=1$, the fusion rules are instead given by \eqref{fr:slLL} alone.
\end{proposition}

\begin{remark}
	Because $(1,1)$ and $(u-1,1)$ are simple currents of $\virminmod{u,v}$, it follows from these fusion rules that the $\sfmod{\ell}{\slirr{1}}$ and $\sfmod{\ell}{\slirr{u-1}}$ are simple currents of the $\slvoa{k}$-module category $\slcat{k}$, for all $\ell \in \ZZ$.
\end{remark}

\begin{proposition}[\protect{\cite[Props.~13 and 18]{CreMod13}}] \label{prop:slgrfus}
	Fix an admissible level $k = -2 + \frac{u}{v}$ and assume that $v>1$.  Given \cref{ass:fusion}, 
	the Grothendieck fusion rules involving the atypicals $\sfmod{\ell}{\sldis{r,s}^+}$ and the typicals $\sfmod{\ell}{\slrel{\lambda}{r,s}}$ then include
	\begin{subequations}
		\begin{align}
			\Gr{\sfmod{\ell}{\slrel{\lambda}{r,s}}} &\Grfus{\slvoa{k}} \Gr{\sfmod{\ell'}{\slrel{\lambda' }{r',s'}}} \notag \\
			= &\sum_{(r'',s'')} \vfusco{r,s}{r',s'}{r'',s''}{u,v} \brac*{\Gr{\sfmod{\ell + \ell' + 1}{\slrel{\lambda + \lambda' - k}{r'',s''}}} + \Gr{\sfmod{\ell + \ell' - 1}{\slrel{\lambda + \lambda' + k}{r'',s''}}}} \notag \\
			+ &\sum_{(r'',s'')} \brac*{\vfusco{r,s}{r',s'-1}{r'',s''}{u,v} + \vfusco{r,s}{r',s'+1}{r'',s''}{u,v}} \Gr{\sfmod{\ell + \ell'}{\slrel{\lambda + \lambda'}{r'',s''}}}, \\
			\Gr{\sfmod{\ell}{\slrel{\lambda}{r,s}}} &\Grfus{\slvoa{k}} \Gr{\sfmod{\ell'}{\sldis{r',s'}^+}} = \sum_{(r'',s'')} \vfusco{r,s}{r',s'+1}{r'',s''}{u,v} \Gr{\sfmod{\ell + \ell'}{\slrel{\lambda + \lambda_{r',s'}}{r'',s''}}} \notag \\
			&\hphantom{\Grfus{\slvoa{k}} \Gr{\sfmod{\ell'}{\sldis{r',s'}^+}} =} + \sum_{(r'',s'')} \vfusco{r,s}{r',s'}{r'',s''}{u,v} \Gr{\sfmod{\ell + \ell' + 1}{\slrel{\lambda + \lambda_{r',s'+1}}{r'',s''}}}.
			\intertext{If $s+s'<v$, then we have in addition}
			\Gr{\sfmod{\ell}{\sldis{r,s}^+}} &\Grfus{\slvoa{k}} \Gr{\sfmod{\ell'}{\sldis{r',s'}^+}} = \sum_{(r'',s'')} \vfusco{r,s}{r',s'}{r'',s''}{u,v} \Gr{\sfmod{\ell + \ell' + 1}{\slrel{\lambda_{r'',s+s'+1}}{r'',s''}}} \notag \\
			&\hphantom{\Grfus{\slvoa{k}} \Gr{\sfmod{\ell'}{\sldis{r',s'}^+}} =} + \sum_{r''=1}^{u-1} \vfusco{r,1}{r',1}{r'',1}{u, v} \Gr{\sfmod{\ell + \ell'}{\sldis{r'',s+s'}^+}},
			\intertext{while if $s+s' \ge v$, then we have instead}
			\Gr{\sfmod{\ell}{\sldis{r,s}^+}} &\Grfus{\slvoa{k}} \Gr{\sfmod{\ell'}{\sldis{r',s'}^+}} = \sum_{(r'',s'')} \vfusco{r,s+1}{r',s'+1}{r'',s''}{u,v} \Gr{\sfmod{\ell + \ell' + 1}{\slrel{\lambda_{r'',s+s'+1}}{r'',s''}}} \notag \\
			&\hphantom{\Grfus{\slvoa{k}} \Gr{\sfmod{\ell'}{\sldis{r',s'}^+}} =} + \sum_{r''=1}^{u-1} \vfusco{r,1}{r',1}{r'',1}{u,v} \Gr{\sfmod{\ell + \ell' + 1}{\sldis{u-r'',s+s'-v+1}^+}}.
		\end{align}
	\end{subequations}
\end{proposition}

\section{The parafermion coset $\pfvoa{k}$} \label{sec:pf}

In this section, we study the parafermion \voa{} $\pfvoa{k} = \Com{\hvoa}{\slvoa{k}}$, where $\hvoa$ denotes the Heisenberg \vosa{} generated by the field $h(z) \in \slvoa{k}$ and the level $k$ is admissible and negative.  We first decompose the characters of the $\slvoa{k}$-modules given in \cref{plainslslchar} into characters of $\hvoa \otimes \pfvoa{k}$-modules.  This relies on the Schur-Weyl duality result summarised in \cref{res:modules}.  We also obtain (Grothendieck) fusion rules for the $\pfvoa{k}$-modules, illustrating the results with examples.

\subsection{Decomposing characters}

We start by recalling that the irreducible modules of the Heisenberg \vosa{} $\hvoa$ are the Fock spaces $\fock{\mu}$, $\mu \in \CC$.  Including the level, and identifying the Heisenberg weight with the $\SLA{sl}{2}$-weight, the characters of the Fock spaces are given by
\begin{equation} \label{eq:fockchar}
	\fch{\fock{\mu}}{y,z,q} = \frac{y^k z^{\mu} q^{\mu^2 / 4k}}{\eta(q)},
\end{equation}
where $\eta$ is Dedekind's eta function.  As the central charge of $\hvoa$ is $1$, that of the parafermion \voa{} $\pfvoa{k}$ is $\tilde{c} = 2 - \frac{6}{t}$.  Denoting the Virasoro zero mode of $\pfvoa{k}$ by $\tilde{L}_0$ (which in the coset realisation is identified with $L_0 - \frac{1}{4k} h_0^2$), the character of a $\pfvoa{k}$-module $\Mod{M}$ is defined to be
\begin{equation}
	\fch{\Mod{M}}{q} = \traceover{\Mod{M}} q^{\tilde{L}_0 - \tilde{c}/24}.
\end{equation}

Since all the irreducible modules of $\slcat{k}$ may be resolved in terms of the standard $\slvoa{k}$-modules $\sfmod{\ell}{\slrel{\lambda}{r,s}}$ and $\sfmod{\ell}{\slindrel{r,s}^+}$, the first natural step is obtain the ``coefficients'' of the Fock space characters in the Schur-Weyl decomposition of these standards.
\begin{prop} \label{cosetEtype}
	Given \cref{ass:cat}, the standard $\slvoa{k}$-modules decompose into $\hvoa \otimes \pfvoa{k}$-modules as
	\begin{subequations}
		\begin{gather}
			\Res{\sfmod{\ell}{\slrel{\lambda}{r,s}}} \cong \bigoplus_{\mu \in \lambda + \rlat} \fock{\mu + \ell k} \otimes \pfrel{\mu}{r,s} \qquad \text{(\(\lambda \neq \lambda_{r,s}, \lambda_{u-r,v-s} \pmod{\rlat}\)),} \label{eq:typdecomp} \\
			\Res{\sfmod{\ell}{\slindrel{r,s}^+}} \cong \bigoplus_{\mu \in \lambda_{r,s} + \rlat} \fock{\mu + \ell k} \otimes \pfindrel{\mu}{r,s}^+, \qquad
			\Res{\sfmod{\ell}{\slindrel{r,s}^-}} \cong \bigoplus_{\mu \in \lambda_{u-r,v-s} + \rlat} \fock{\mu + \ell k} \otimes \pfindrel{\mu}{r,s}^-,
		\end{gather}
	\end{subequations}
	where the $\pfrel{\mu}{r,s}$ are irreducible \hw{} $\pfvoa{k}$-modules and the $\pfindrel{\mu}{r,s}^{\pm}$ are length $2$ indecomposable $\pfvoa{k}$-modules.  Their characters are given by
	\begin{equation} \label{eq:pfrelchar}
			\fch{\pfrel{\mu}{r,s}}{q} = \fch{\pfindrel{\mu}{r,s}^{\pm}}{q} = \frac{q^{-\mu^2 / 4k} \vch{u,v}{r,s}{q}}{\eta(q)}.
	\end{equation}
\end{prop}
\begin{proof}
	Schur-Weyl duality  (\cref{res:modules}) immediately implies that
	\begin{equation}
		\Res{\slrel{\lambda}{r,s}} \cong \bigoplus_{\mu \in \lambda + \rlat} \fock{\mu} \otimes \pfrel{\mu}{r,s},
	\end{equation}
	where the $\pfrel{\mu}{r,s}$ are irreducible $\pfvoa{k}$-modules.  We obtain the parafermion characters by decomposing \eqref{eq:typchar} (with $\ell = 0$):
	\begin{align} \label{eq:TypDecomp}
		\ch{\slrel{\lambda}{r,s}}
		&= \frac{y^k \vch{u,v}{r,s}{q}}{\eta(q)^2} \sum_{n \in \ZZ} z^{2n + \lambda}
		= \sum_{n \in \ZZ} \frac{y^k z^{2n + \lambda} q^{(2n + \lambda)^2 / 4k}}{\eta(q)} \cdot \frac{q^{-(2n + \lambda)^2 / 4k} \vch{u,v}{r,s}{q}}{\eta(q)} \notag \\
		&= \sum_{n \in \ZZ} \ch{\fock{2n + \lambda}} \frac{q^{-(2n + \lambda)^2 / 4k} \vch{u,v}{r,s}{q}}{\eta(q)}.
	\end{align}
	The desired result now follows by identifying $2n + \lambda$ with $\mu \in \lambda + 2 \ZZ = \lambda + \rlat$.

	Analogous decompositions hold for the $\sfmod{\ell}{\slrel{\lambda}{r,s}}$, so it remains to show that the parafermion modules appearing in these decompositions may be identified with the $\pfrel{\mu}{r,s}$.  This follows from the lifting condition (\cref{res:lifting}).  Indeed, this condition guarantees that there exists $\beta \in \rlat \otimes_{\ZZ} \CC = \CC$ such that $\fock{\nu} \otimes \pfrel{\mu}{r,s}$ lifts to an $\slvoa{k}$-module if and only if $\nu \in \beta + \rlat'$, where $\rlat' = k \ZZ$ is the dual lattice of $\rlat = 2 \ZZ$ (with respect to the natural bilinear form $\inner{h}{h} = 2k$ induced by the \ope{} of $h(z)$ and $h(w)$).  As $\Ind{\brac[\big]{\fock{\mu} \otimes \pfrel{\mu}{r,s}}} \cong \slrel{\lambda}{r,s}$ is an $\slvoa{k}$-module, we may take $\beta = \mu$.

	It follows that for any $\ell \in \ZZ$, the lift $\Ind{\brac[\big]{\fock{\mu + \ell k} \otimes \pfrel{\mu}{r,s}}} = \bigoplus_{\mu \in \lambda} \fock{\mu + \ell k} \otimes \pfrel{\mu}{r,s}$ is an $\slvoa{k}$-module and it is irreducible because $\pfrel{\mu}{r,s}$ is (by Schur-Weyl duality).  A calculation very similar to \eqref{eq:TypDecomp} shows that this irreducible module has the same character as $\sfmod{\ell}{\slrel{\lambda}{r,s}}$, hence they are isomorphic.  This therefore establishes the decomposition \eqref{eq:typdecomp} for all $\ell \in \ZZ$.  The results for the atypical standard modules $\sfmod{\ell}{\slindrel{r,s}^{\pm}}$ follow using similar arguments, the main difference being that the isomorphisms follow from the indecomposables being completely characterised by their Loewy diagrams (the extension groups are $1$-dimensional).
\end{proof}

\begin{remark}\label{rem:noiso}
	We point out that this result used the fact that the irreducible $\slvoa{k}$-modules are determined (up to isomorphism) by their characters.  The same is unfortunately not true for characters derived above for the standard $\pfvoa{k}$-modules.  For example, $\pfrel{\mu}{r,s}$ and $\pfrel{-\mu}{r,s}$ share the same character, despite being non-isomorphic (for $\mu \neq 0$). We can see this inequivalence as follows. First, as $\fock{\mu} \otimes \pfrel{\mu}{r,s}$ lifts to the $\slvoa{k}$-module $\slrel{\mu}{r,s}$, the lifting condition (\cref{res:lifting}) says that the tensor product $\fock{\nu} \otimes \pfrel{\mu}{r,s}$ lifts to an $\slvoa{k}$-module if and only if $\nu \in \mu + \rlat'$.  If $\pfrel{\mu}{r,s}$ and $\pfrel{-\mu}{r,s}$ were isomorphic, then there would have to exist $\nu \in (\mu + \rlat') \cap (-\mu + \rlat')$ which is empty unless $2 \mu \in \rlat'$.  It follows that $\pfrel{\mu}{r,s}$ and $\pfrel{-\mu}{r,s}$ are not isomorphic if $2\mu \notin \rlat'$.  But, if $2\mu \in \rlat' = k \ZZ$, then $\mu = -\mu + \ell k$, for some $\ell \in \ZZ$, and any isomorphism $\pfrel{\mu}{r,s} \cong \pfrel{-\mu}{r,s}$ would lead to
	\begin{equation}
		\slrel{\mu}{r,s} \cong \Ind{(\fock{\mu} \otimes \pfrel{\mu}{r,s})} \cong \Ind{(\fock{\mu} \otimes \pfrel{-\mu}{r,s})} \cong \Ind{(\fock{-\mu + \ell k} \otimes \pfrel{-\mu}{r,s})} \cong \sfmod{\ell}{\slrel{-\mu}{r,s}},
	\end{equation}
	which is a contradiction unless $\ell = 0$, hence $\mu = 0$.
\end{remark}

Combining this result with the character formulae given for the atypical irreducibles in \cref{plainslslchar} and/or the resolutions of \cref{rem:slres}, we deduce the latter's decompositions into $(\hvoa \otimes \pfvoa{k})$-modules and character formulae for the resulting parafermion modules.

\begin{prop} \label{charLcoset}
	Given \cref{ass:cat}, the atypical irreducible $\slvoa{k}$-modules decompose into $(\hvoa \otimes \pfvoa{k})$-modules as
	\begin{equation} \label{eq:atypdecomp}
		\Res{\sfmod{\ell}{\slirr{r}}} \cong \bigoplus_{\mu \in \lambda_{r,0} + \rlat} \fock{\mu + \ell k} \otimes \pfirr{\mu}{r}, \qquad
		\Res{\sfmod{\ell}{\sldis{r,s}^+}} \cong \bigoplus_{\mu \in \lambda_{r,s} + \rlat} \fock{\mu + \ell k} \otimes \pfdis{\mu}{r,s},
	\end{equation}
	where the $\pfirr{\mu}{r}$ and $\pfdis{\mu}{r,s}$ are irreducible \hw{} $\pfvoa{k}$-modules characterised by the following resolutions:
	\begin{subequations} \label{eq:pfres}
		\begin{align} \label{eq:pfresL}
			\cdots &\lra \pfindrel{\mu-(3v-1)k}{r,v-1}^+ \lra \cdots \lra \pfindrel{\mu-(2v+2)k}{r,2}^+ \lra \pfindrel{\mu-(2v+1)k}{r,1}^+ \notag \\
			&\qquad\qquad \lra \pfindrel{\mu-(2v-1)k}{u-r,v-1}^+ \lra \cdots \lra \pfindrel{\mu-(v+2)k}{u-r,2}^+ \lra \pfindrel{\mu-(v+1)k}{u-r,1}^+ \notag \\
			&\qquad\qquad\qquad\qquad \lra \pfindrel{\mu-(v-1)k}{r,v-1}^+ \lra \cdots \lra \pfindrel{\mu-2k}{r,2}^+ \lra \pfindrel{\mu-k}{r,1}^+ \lra \pfirr{\mu}{r} \lra 0,
		\end{align}
		\begin{equation} \label{eq:pfresD}
			0 \lra \pfirr{\mu-(v-s)k}{u-r} \lra \pfindrel{\mu-(v-1-s)k}{r,v-1}^+ \lra \cdots \lra \pfindrel{\mu-2k}{r,s+2}^+ \lra \pfindrel{\mu-k}{r,s+1}^+ \lra \pfdis{\mu}{r,s} \lra 0.
		\end{equation}
	\end{subequations}
	Their characters are given by
	\begin{subequations} \label{eq:pfatypchar}
		\begin{align}
			\fch{\pfirr{\mu}{r}}{q} &= \sum_{s=1}^{v-1} (-1)^{s-1} \frac{\vch{u,v}{r,s}{q}}{\eta(q)} \sum_{m=0}^{\infty} \brac*{q^{-(\mu-sk+2wm)^2 / 4k} - q^{-(\mu+sk+2w(m+1))^2 / 4k}}, \label{eq:pfirrchar} \\
			\fch{\pfdis{\mu}{r,s}}{q} &= \sum_{s'=s+1}^{v-1} (-1)^{s'-s-1} \frac{\vch{u,v}{r,s'}{q}}{\eta(q)} q^{-(\mu-(s'-s)k)^2 / 4k} + (-1)^{v-1-s} \ch{\pfirr{\mu-(v-s)k}{u-r}}. \label{eq:pfdischar}
		\end{align}
	\end{subequations}
\end{prop}

\begin{remark} \label{rem:D=L}
	It is easy to check from these formulae that $\pfdis{\mu}{r,v-1}$ and $\pfirr{\mu - k}{u-r}$ have the same character, consistent with the fact that these $\pfvoa{k}$-modules are isomorphic (by \eqref{eq:pfresD} with $s=v-1$).
\end{remark}

\begin{remark}
	As an alternative to the resolutions \eqref{eq:pfres}, we present the following non-split short exact sequences that characterise the atypical standard $\pfvoa{k}$-modules:
	\begin{subequations}
		\begin{align} \label{es:pfDED}
			\dses{\pfdis{\mu}{r,s}}{\pfindrel{\mu}{r,s}^+}{\pfdis{\mu+k}{r,s-1}}& & &\text{(\(s \neq 1\)),} \\
			\dses{\pfdis{\mu+k}{u-r,v-1-s}}{\pfindrel{\mu}{r,s}^-}{\pfdis{\mu}{u-r,v-s}}& & &\text{(\(s \neq v-1\)).}
		\end{align}
	\end{subequations}
	When $s=1$, the rightmost module of the first sequence should be replaced by $\pfirr{\mu+k}{r}$.  Similarly, the leftmost module of the second sequence should be replaced by $\pfirr{\mu+k}{u-r}$ when $s=v-1$.
\end{remark}

\begin{remark} \label{rem:pfirraltchar}
	If we had started with the resolution obtained from \cref{eq:slresL} by conjugating, then we would have instead arrived at the following character formula:
	\begin{equation}
		\fch{\pfirr{\mu}{r}}{q} = \sum_{s=1}^{v-1} (-1)^{s-1} \frac{\vch{u,v}{r,s}{q}}{\eta(q)} \sum_{m=0}^{\infty} \brac*{q^{-(\mu+sk+2vkm)^2 / 4k} - q^{-(\mu-sk+2vk(m+1))^2 / 4k}}.
	\end{equation}
	Replacing $\mu$ by $-\mu$ and comparing with \eqref{eq:pfirrchar}, we conclude that $\pfirr{\mu}{r}$ and $\pfirr{-\mu}{r}$ have the same character, despite being non-isomorphic (for $\mu \neq 0$). The reason is exactly the same argument as in Remark \ref{rem:noiso}.
\end{remark}

The natural category of $\pfvoa{k}$-modules to consider is thus the full subcategory whose simple objects are the irreducibles $\pfirr{\mu}{r}$, $\pfdis{\mu}{r,s}$ and $\pfrel{\mu}{r,s}$ and whose objects are all realised as subquotients of the fusion product of a finite collection of simple objects.  We shall assume that no further irreducible $\pfvoa{k}$-modules are generated as subquotients of such fusion products, hence that this category does indeed exist.  It shall be denoted by $\pfcat{k}$.

\begin{remark}
	We do not claim that the $\pfirr{\mu}{r}$, $\pfdis{\mu}{r,s}$ and $\pfrel{\mu}{r,s}$ exhaust the irreducible $\pfvoa{k}$-modules because of the Whittaker-type $\slvoa{k}$-modules mentioned in \cref{rem:whittaker}.  Indeed, $\hvoa$ acts non-semisimply on these $\slvoa{k}$-modules, so we cannot use \cref{res:modules} to easily check if decomposing into $\hvoa \otimes \pfvoa{k}$-modules gives anything new.
\end{remark}

For future convenience, we collect the conformal weights of the ground states of these $\pfvoa{k}$-modules.  For the $\pfvoa{k}$-module denoted by $\pfmod{\mu}{\star}^{\bullet}$ (for appropriate $\bullet$ and $\star$), this weight will be denoted by $\pfdim{\mu}{\star}{\bullet}$.
\begin{proposition} \label{prop:pfconfwts}
	The conformal weights of the ground states of the $\pfvoa{k}$-modules introduced above are
	\begin{subequations}
		\begin{align}
			\pfreldim{\mu}{r,s} &= \pfdim{\mu}{r,s}{\pm} = \Delta_{r,s} - \frac{\mu^2}{4k}, \label{eq:pfrelconfwts} \\
			\pfirrdim{\mu}{r} &=
			\begin{dcases*}
				\Delta_{r,0} - \frac{\mu^2}{4k} & if \(\abs{\mu} \le \lambda_{r,0}\), \\
				\Delta_{r,0} - \frac{\mu^2}{4k} + \frac{\abs{\mu} - \lambda_{r,0}}{2} & if \(\abs{\mu} \ge \lambda_{r,0}\),
			\end{dcases*}
			\label{eq:pfirrconfwts} \\
			\pfdisdim{\mu}{r,s} &=
			\begin{dcases*}
				\Delta_{r,s} - \frac{\mu^2}{4k} & if \(\mu \le \lambda_{r,s}\), \\
				\Delta_{r,s} - \frac{\mu^2}{4k} + \frac{\mu - \lambda_{r,s}}{2} & if \(\mu \ge \lambda_{r,s}\),
			\end{dcases*}
			\qquad\qquad \text{(\(s \neq v-1\)).} \label{eq:pfdisconfwts}
		\end{align}
	\end{subequations}
	For $s=v-1$, we have instead $\pfdisdim{\mu}{r,v-1} = \pfirrdim{\mu-k}{u-r}$, by \cref{rem:D=L}.
\end{proposition}
\begin{proof}
	These results follow easily from the observation that any state of minimal conformal weight in an $h_0$-eigenspace of an $\AKMA{sl}{2}$-module will be a \hwv{} for both the Heisenberg and parafermion \vosas{}.  The conformal weight of such a state, with respect to $\pfvoa{k}$, is then its $\AKMA{sl}{2}$ conformal weight minus its Heisenberg conformal weight.
\end{proof}

\subsection{Fusion}

Recall that we only know the (Grothendieck) fusion rules of $\slvoa{k}$ up to the validity of \cref{ass:fusion}.  To 
deduce the (Grothendieck) fusion rules of $\pfvoa{k}$ from those of $\slvoa{k}$ using \cref{res:fus,cosetEtype,charLcoset}, we also need \cref{ass:cat}, namely that the category $\pfcat{k}$ of $\pfvoa{k}$-modules forms a vertex tensor category.  We recall that this \lcnamecref{ass:cat} is in force throughout.

To illustrate the method, consider the $\slvoa{k}$ fusion rule \eqref{fr:slLL}.  Applying \cref{charLcoset,res:fus}, we deduce that
\begin{align} \label{fr:pfill}
	\Ind{\brac*{\fock{\mu} \otimes \pfirr{\mu}{r}}} \fus{\slvoa{k}} \Ind{\brac*{\fock{\mu'} \otimes \pfirr{\mu'}{r'}}}
	&\cong \bigoplus_{r''=1}^{u-1} \vfusco{r,1}{r',1}{r'',1}{u,v} \Ind{\brac*{\fock{\mu''} \otimes \pfirr{\mu''}{r''}}} \notag \\
	\implies \quad \Ind{\brac*{\fock{\mu+\mu'} \otimes \brac*{\pfirr{\mu}{r} \fus{\pfvoa{k}} \pfirr{\mu'}{r'}}}}
	&\cong \Ind{\brac*{\fock{\mu''} \otimes \bigoplus_{r''=1}^{u-1} \vfusco{r,1}{r',1}{r'',1}{u,v} \pfirr{\mu''}{r''}}},
\end{align}
whenever $\mu = r-1 \pmod{\rlat}$, $\mu' = r'-1 \pmod{\rlat}$ and $\mu'' = r''-1 \pmod{\rlat}$.  It follows that the $\hvoa \otimes \pfvoa{k}$-modules being induced on the left- and right-hand sides of \eqref{fr:pfill} appear as direct summands of the same $\slvoa{k}$-module upon restricting to $\hvoa \otimes \pfvoa{k}$.  As these direct summands are completely determined by their Heisenberg weights, we can read off the $\pfvoa{k}$ fusion rule by identifying $\mu''$ with $\mu + \mu'$.

\begin{proposition} \label{prop:pffusL}
	Given \cref{ass:cat,ass:fusion}, the fusion rules of the irreducible $\pfvoa{k}$-modules with the $\pfirr{\mu}{r}$ are
	\begin{subequations}
		\begin{align}
			\pfirr{\mu}{r} \fus{\pfvoa{k}} \pfirr{\mu'}{r'} &\cong \bigoplus_{r''=1}^{u-1} \vfusco{r,1}{r',1}{r'',1}{u,v} \pfirr{\mu+\mu'}{r''}, \label{fr:pfLL} \\
			\pfirr{\mu}{r} \fus{\pfvoa{k}} \pfdis{\mu'}{r',s'} &\cong \bigoplus_{r''=1}^{u-1} \vfusco{r,1}{r',1}{r'',1}{u,v} \pfdis{\mu+\mu'}{r'',s'}, \label{fr:pfLD} \\
			\pfirr{\mu}{r} \fus{\pfvoa{k}} \pfrel{\mu'}{r',s'} &\cong \bigoplus_{r''=1}^{u-1} \vfusco{r,1}{r',1}{r'',1}{u,v} \pfrel{\mu+\mu'}{r'',s'}. \label{fr:pfLE}
		\end{align}
	\end{subequations}
\end{proposition}

\begin{remark}
	Note that the $\pfirr{\mu}{1}$ with $\mu \in \rlat$ and the $\pfirr{\mu}{u-1}$ with $\mu \in u + \rlat$ are all simple currents in the $\pfvoa{k}$-module category $\pfcat{k}$.  The vacuum $\pfvoa{k}$-module $\pfirr{0}{1}$ is the fusion unit, as expected.  Excluding the vacuum module, the simple currents of minimal conformal weight are either $\pfirr{\pm 2}{1}$ or one of $\pfirr{0}{u-1}$ and $\pfirr{\pm 1}{u-1}$, according as to whether $u$ is even or odd, respectively.  These minimal conformal weights are
	\begin{equation} \label{eq:pfsimcurconfwts}
		\pfirrdim{\pm 2}{1} = 1+\frac{v}{w}, \qquad
		\pfirrdim{0}{u-1} = \frac{(u-2)v}{4}, \qquad
		\pfirrdim{\pm 1}{u-1} = \frac{(u-2)v}{4} + \frac{v}{4w},
	\end{equation}
	by \cref{prop:pfconfwts} (recall that $w = 2v-u$).  The order of $\pfirr{0}{u-1}$ is $2$, assuming that $u>2$, whilst the other (non-vacuum) simple currents all have infinite orders.
\end{remark}

\begin{proposition} \label{prop:pfgrfus}
	Given \cref{ass:cat,ass:fusion}, the Grothendieck fusion rules involving the atypicals $\pfdis{\mu}{r,s}$ and the typicals $\pfrel{\mu}{r,s}$ include
	\begin{subequations}
		\begin{align}
			\Gr{\pfrel{\mu}{r,s}} \Grfus{\pfvoa{k}} \Gr{\pfrel{\mu' }{r',s'}}
			&= \sum_{(r'',s'')} \vfusco{r,s}{r',s'}{r'',s''}{u,v} \brac*{\Gr{\pfrel{\mu + \mu'-k}{r'',s''}} + \Gr{\pfrel{\mu + \mu' + k}{r'',s''}}} \notag \\
			&\mspace{-30mu} + \sum_{(r'',s'')} \brac*{\vfusco{r,s}{r',s'-1}{r'',s''}{u,v} + \vfusco{r,s}{r',s'+1}{r'',s''}{u,v}} \Gr{\pfrel{\mu + \mu'}{r'',s''}}, \\
			\Gr{\pfrel{\mu}{r,s}} \Grfus{\pfvoa{k}} \Gr{\pfdis{\mu'}{r',s'}} &= \sum_{(r'',s'')} \vfusco{r,s}{r',s'+1}{r'',s''}{u,v} \Gr{\pfrel{\mu + \mu'}{r'',s''}} \notag \\
			&\mspace{100mu} + \sum_{(r'',s'')} \vfusco{r,s}{r',s'}{r'',s''}{u,v} \Gr{\pfrel{\mu+\mu'-k}{r'',s''}}.
			\intertext{If $s+s'<v$, then we have in addition}
			\Gr{\pfdis{\mu}{r,s}} \Grfus{\pfvoa{k}} \Gr{\pfdis{\mu'}{r',s'}} &= \sum_{(r'',s'')} \vfusco{r,s}{r',s'}{r'',s''}{u,v} \Gr{\pfrel{\mu+\mu'-k}{r'',s''}} \notag \\
			&\mspace{100mu}+ \sum_{r''=1}^{u-1} \vfusco{r,1}{r',1}{r'',1}{u, v} \Gr{\pfdis{\mu+\mu'}{r'',s+s'}},
			\intertext{while if $s+s' \ge v$, then we have instead}
			\Gr{\pfdis{\mu}{r,s}} \Grfus{\pfvoa{k}} \Gr{\pfdis{\mu'}{r',s'}} &= \sum_{(r'',s'')} \vfusco{r,s+1}{r',s'+1}{r'',s''}{u,v} \Gr{\pfrel{\mu+\mu'-k}{r'',s''}} \notag \\
			&\mspace{100mu} + \sum_{r''=1}^{u-1} \vfusco{r,1}{r',1}{r'',1}{u,v} \Gr{\pfdis{\mu+\mu'-k}{u-r'',s+s'-v+1}}.
		\end{align}
	\end{subequations}
\end{proposition}

\begin{remark}
One can also determine the modular transformations of the characters of the irreducible $\pfvoa{k}$-modules and check that the standard Verlinde formula reproduces the Grothendieck fusion rules given here.  These modular properties may either be computed directly or, more easily, from the known modular properties of the irreducible $\slvoa{k}$- and $\hvoa$-modules.  We will not pursue these straightforward computations here.  Instead, we shall study the much more interesting modular properties of an infinite-order simple current extension of $\pfvoa{k}$, in \cref{sec:epf}.
\end{remark}

\subsection{Examples} \label{pfexamples}

The parafermion coset construction for the levels $k=-\frac{1}{2}$ and $k=-\frac{4}{3}$ has already been discussed in detail \cite{AdaCon05,RidSL210} with the result being that $\pfvoa{-1/2}$ and $\pfvoa{-4/3}$ may be identified as the well known singlet \voas{} $\singvoa{1,2}$ and $\singvoa{1,3}$ of central charges $\tilde{c} = -2$ and $\tilde{c} = -7$, respectively.  The decomposition of $\slvoa{k}$-modules into $\pfvoa{k}$-modules is given very explicitly in \cite[Sec.~4]{CreCos13} and singlet fusion rules have been computed (within the conjectural standard module formalism) in \cite{CreLog13,Crefal13,RidMod13}.  A rigorous computation of certain fusion coefficients for the $p=2$ singlet has also recently appeared \cite{Adafus16}.  All these computations are consistent with the (Grothendieck) fusion rules reported here.  We add that for these levels, much is known about the category $\pfcat{k}$ of $\pfvoa{k}$-modules.  In particular, it is a vertex tensor category provided that a $C_1$-cofiniteness condition and a finite Jordan-H\"older length condition hold, see \cite[Thm.~17]{CreQG16}.

An important family of parafermion cosets $\pfvoa{k}$ are those with $k=-\frac{n-1}{n}$, for $n\in \ZZ_{\ge 2}$.  Since $u=n+1$ and $v=n$ in these cases, the Virasoro characters appearing in the parafermion characters and the Virasoro fusion coefficients appearing in the parafermion fusion rules are those of the unitary Virasoro minimal models.  One reason for this importance is their relation to the \svoas{} $\savoa{k'}{\SLSA{sl}{2}{1}}$ with $k'$ a positive integer \cite{FeiW2n04}.  We intend to report on this in the near future. A second reason is that it is also expected that $C_2$-cofinite extensions of orbifolds of this family of parafermion cosets $\pfvoa{k}$ coincides with certain cosets of the minimal $\VOA{W}$-algebras of $\SLA{so}{2(n+1)}$ at level zero \cite{AraMin17}.

Here, we list the inequivalent irreducible $\pfvoa{-(n-1)/n}$-modules explicitly, recalling \cref{rem:D=L}:
\begin{itemize}
\item the $\pfirr{\mu}{r}$ with $r = 1, \dots, n$ and $\mu \in 2\ZZ+r-1$;
\item the $\pfdis{\mu}{r,s}$ with $r = 1, \dots, n$, $s = 1, \dots, n-2$ and $\mu + \frac{s}{n} \in 2\ZZ+r+s-1$;
\item the $\pfrel{\mu}{r,s}$ with either $\mathrlap{r = 1, \dots, \frac{n}{2}} \hphantom{r=1, \dots, n}$, $s = 1, \dots, n-1$ and $\mu \pm \frac{s}{n} \notin 2\ZZ+r+s-1$, if $n$ is even, or \\ \hphantom{the $\pfrel{\mu}{r,s}$ with either} $r = 1, \dots, n$, $\mathrlap{s = 1, \dots, \frac{n-1}{2}} \hphantom{s = 1, \ldots, n-1}$ and $\mu \pm \frac{s}{n} \notin 2\ZZ+r+s-1$, if $n$ is odd.
\end{itemize}
In terms of conciseness, it would be convenient to replace the $\pfirr{\mu}{r}$ by the $\pfdis{\mu}{r,n-1}$ in the above list.  However, we prefer to distinguish the $\slirr{}$-type modules explicitly as they include all the simple currents (including the vacuum module).

The first member, $n=2$, of this family of \voas{} is, as was mentioned above, the singlet $\pfvoa{-1/2} \cong \singvoa{1,2}$.  The above list recovers the known \cite{CreLog13} module spectrum.  Specifically, there are two series of $\slirr{}$-type modules $\pfirr{2m}{1}$ and $\pfirr{2m+1}{2}$, $m \in \ZZ$, which are all simple currents, no (inequivalent) $\sldis{}$-type modules and one series of $\slindrel{}$-type modules $\pfrel{\mu}{1,1}$, $\mu \in \ZZ + \frac{1}{2}$.  The minimal conformal weight is $\Delta_{1,1} = -\frac{1}{8}$.

A more interesting (and less familiar) example is $n=3$, thus $k=-\frac{2}{3}$ and $\tilde{c} = -\frac{5}{2}$.  This central charge matches that of the $N=1$ logarithmic superconformal minimal model $\nslogminmod{1,3}$ \cite{PeaLog14,CanFusI15,CanFusII15}.  Up to isomorphism, we now have six families of irreducible atypicals
\begin{subequations} \label{pf-2/3irreps}
	\begin{equation}
		\begin{gathered}
			\pfirr{2m}{1} \cong \pfdis{2m-2/3}{3,2} \quad \pfirr{2m+1}{2} \cong \pfdis{2m+1/3}{2,2} \quad \pfirr{2m}{3} \cong \pfdis{2m-2/3}{1,2} \\
			\pfdis{2m+2/3}{1,1} \qquad \pfdis{2m-1/3}{2,1} \qquad \pfdis{2m+2/3}{3,1}
		\end{gathered}
		\qquad \text{(\(m \in \ZZ\))}
	\end{equation}
	and three families of typicals
	\begin{equation}
		\pfrel{\mu}{1,1} \cong \pfrel{\mu}{3,2} \qquad \pfrel{\mu}{2,1} \cong \pfrel{\mu}{2,2} \qquad \pfrel{\mu}{3,1} \cong \pfrel{\mu}{1,2},
	\end{equation}
\end{subequations}
where $\mu \notin 2\ZZ \pm \frac{2}{3}$, $\mu \notin 2\ZZ \pm \frac{1}{3}$ and $\mu \notin 2\ZZ \pm \frac{2}{3}$, respectively.  The simple currents are the $\pfirr{2m}{1}$ and $\pfirr{2m}{3}$; their conformal weights are $\pfirrdim{2m}{1} = \frac{3}{2} \abs{m} (\abs{m}+1)$, $\pfirrdim{2m}{3} = \frac{3}{2} (\abs{m}+\frac{1}{3})^2 + \frac{1}{3}$, if $m \neq 0$, and $\pfirrdim{0}{3} = \frac{3}{2}$.

The conformal weight $\frac{3}{2}$ simple current $\Mod{G} = \pfirr{0}{3}$ has order $2$:  $\Mod{G} \fus{\pfvoa{-2/3}} \Mod{G} \cong \pfvoa{-2/3}$.  As the dimension of the space of ground states of $\Mod{G}$ is $1$, the corresponding simple current extension of $\pfvoa{-2/3}$ contains precisely one copy of the \svoa{} $\nslogminmod{1,3}$.  We denote this extension by $\spfvoa{-2/3}$ so that
\begin{equation}
	\Res{\spfvoa{-2/3}} \cong \pfvoa{-2/3} \oplus \Mod{G}.
\end{equation}

The character of this extended parafermionic \svoa{} is easy to determine using \eqref{eq:pfirrchar}:
\begin{align}
	\fch{\spfvoa{-2/3}}{q} &= \fch{\pfirr{0}{1}}{q} + \fch{\pfirr{0}{3}}{q} \notag \\
	&= \frac{\vch{4,3}{1,1}{q} + \vch{4,3}{1,2}{q}}{\eta(q)} \sum_{m=0}^{\infty} \sqbrac*{q^{3(2m+2/3)^2/8} - q^{3(2m+4/3)^2/8}} \notag \\
	&= \frac{1}{\eta(q)} \sqrt{\frac{\fjth{3}{1,q}}{\eta(q)}} \sum_{m=0}^{\infty} \sqbrac*{q^{(3m+1)^2/6} - q^{(3m+2)^2/6}} \\
	&= q^{-\tilde{c}/24} \brac*{1 + q^{3/2} + q^2 + 2 q^{5/2} + 2 q^3 + 3 q^{7/2} + 4 q^4 + 5 q^{9/2} + 6 q^5 + \cdots}. \notag
\end{align}
This shows that $\spfvoa{-2/3} \ncong \nslogminmod{1,3}$ because the coefficients of $q^{5/2}$ and $q^3$ in the latter's character are only $1$.  However, this extra state of conformal weight $\frac{5}{2}$ leads us to the decomposition
\begin{equation}
	\fch{\spfvoa{-2/3}}{q} = \sum_{m=0}^{\infty} \nsch{1,3}{2m+1,1}{q} = \fch{\ssingvoa{1,3}}{q},
\end{equation}
where $\nsch{1,3}{2m+1,1}{q}$ denotes the character of the irreducible $\nslogminmod{1,3}$-module whose \hwv{} has conformal weight $\nsdim{1,3}{2m+1,1} = \frac{1}{2} m (3m+2)$ and $\ssingvoa{1,3}$ denotes the $N=1$ singlet \svoa{} \cite{AdaN=109,AdaN=108} of central charge $-\frac{5}{2}$.

It is now straightforward to verify, with the aid of a computer, that the operator product algebras of the bosonic orbifold of the $N=1$ supersinglet $\ssingvoa{1,3}$ and the parafermion \voa{} $\pfvoa{-2/3}$ coincide. Since both \voas{} are simple, they must be isomorphic. We therefore identify the parafermion \voa{} $\pfvoa{-2/3}$ as the bosonic orbifold of the $N=1$ supersinglet $\ssingvoa{1,3}$ and its simple current extension $\spfvoa{-2/3}$ as $\ssingvoa{1,3}$.

We conclude this example by considering the induction of $\pfvoa{-2/3}$-modules $\Mod{M}$ to (twisted) $\spfvoa{-2/3}$-modules via
\begin{equation}
	\Ind{\Mod{M}} = \spfvoa{-2/3} \fus{\pfvoa{-2/3}} \Mod{M} \quad \implies \quad \Res{(\Ind{\Mod{M}})} \cong \Mod{M} \oplus \brac*{\Mod{G} \fus{\pfvoa{-2/3}} \Mod{M}}.
\end{equation}
Whether the resulting module is twisted or not (Ramond or \ns{}) depends only on the difference mod $\ZZ$ of the conformal weights of $\Mod{M}$ and $\Mod{G} \fus{\pfvoa{-2/3}} \Mod{M}$.  Noting that the fusion rules of \cref{prop:pffusL} specialise to
\begin{equation}
	\Mod{G} \fus{\pfvoa{-2/3}} \pfirr{\mu}{r} \cong  \pfirr{\mu}{4-r},  \qquad
	\Mod{G} \fus{\pfvoa{-2/3}} \pfdis{\mu}{r,1} \cong  \pfdis{\mu}{4-r,1}, \qquad
	\Mod{G} \fus{\pfvoa{-2/3}} \pfrel{\mu}{r,1} \cong  \pfrel{\mu}{4-r,1},
\end{equation}
it easy to check that this difference is $\frac{r}{2} \pmod{\ZZ}$.  The nine families of irreducible $\pfvoa{-2/3}$-modules listed in \eqref{pf-2/3irreps} therefore lift to six families of irreducible $\spfvoa{-2/3}$-modules in the \ns{} sector ($r$ odd), three of which are just parity-reversed copies of the other three, and three families of irreducible $\spfvoa{-2/3}$-modules in the Ramond sector ($r$ even), each of which is isomorphic to its parity reversal.

\begin{remark}
	In the early W-algebra literature, the \svoa{} $\ssingvoa{1,3}$ is referred to as the $N=1$ super-$\VOA{W}_3$ algebra because the additional fields of conformal weights $\frac{5}{2}$ and $3$ naturally form a superfield generalising the weight $3$ field of the Casimir W-algebra of $\SLA{sl}{3}$.  It was one of the first examples found of an ``exotic'' W-algebra, meaning that it only exists for a discrete set of central charges, in this case $-\frac{5}{2}$ and $\frac{10}{7}$ \cite{InaExt88}.  The latter central charge received much attention, see \cite{BouWSym93} and references therein, as it corresponds to a unitary value for both the $N=1$ superconformal minimal models and the Casimir W-algebras of $\SLA{sl}{3}$.  However, we are not aware of any detailed study of the non-unitary (and in fact logarithmic) $\tilde{c} = -\frac{5}{2}$ super-$\VOA{W}_3$ algebra in the literature.
\end{remark}

\section{The extended parafermion coset $\epfvoa{k}$} \label{sec:epf}

We now study a larger \voa{} $\epfvoa{k}$ as a coset of an extension of $\slvoa{k}$ or, as advocated in \eqref{pic:thegame} of the introduction, as a simple current extension of $\pfvoa{k}$.  As we shall show, the $\epfvoa{k}$ are not rational, because they admit reducible but indecomposable modules, but have a finite number of irreducibles, up to isomorphism.  Moreover, the characters of the irreducible $\epfvoa{k}$-modules will be shown to define a finite-dimensional vector-valued modular form.  We therefore conjecture that the $\epfvoa{k}$ are $C_2$-cofinite.  As in the previous section, the level $k$ will be assumed throughout to be admissible and negative.  Throughout, \cref{ass:cat,ass:fusion} are understood to be in force.

\subsection{$\epfvoa{k}$ as a coset and a simple current extension}

Recall that the vacuum module of $\slvoa{k}$ is $\slirr{1}$ and that its images $\sfmod{\ell}{\slirr{1}}$ under spectral flow are simple currents.  We consider the module	$\bigoplus_{\ell \in\ZZ} \sfmod{\alpha \ell}{\slirr{1}}$, where $\alpha \in \ZZ$.  It is easy to check that this simple current extension of $\slvoa{k}$ will be a \voa{} if and only if it is $\ZZ$-graded, which happens if and only if $k\alpha^2 \in 4\ZZ$.  This implies that $\alpha$ needs to be an integer multiple of $v$ and a convenient choice that works for all admissible levels is $\alpha=2v$.  We thus define the \voa{} $\eslvoa{k}$ as the simple current extension which decomposes into $\slvoa{k}$-modules as follows:
\begin{equation} \label{eq:DefE}
	\Res{\eslvoa{k}} \cong \bigoplus_{\ell \in\ZZ} \sfmod{2v \ell}{\slirr{1}}.
\end{equation}
We remark that because the conformal weights of $\sfmod{\ell}{\slirr{1}}$ are not bounded below whenever $\abs{\ell} \ge 2$, the \voa{} $\eslvoa{k}$ is not $\NN$-graded by its conformal weights.

Inserting the decomposition of \cref{charLcoset}, \cref{eq:DefE} becomes
\begin{equation} \label{eq:DecompE'}
	\Res{\eslvoa{k}} \cong \bigoplus_{\ell \in\ZZ} \bigoplus_{\mu \in \rlat} \fock{\mu +2vk \ell} \otimes \pfirr{\mu}{1}.
\end{equation}
Consider the lattice \voa{}
\begin{equation} \label{eq:DefL}
	\lvoa{L} = \bigoplus_{\lambda \in \lat{L}} \fock{\lambda},
\end{equation}
where $\lat{L} = -2vk \ZZ = w \rlat$ (recall that $w=-vk=2v-u$).  By considering the modules with $\mu = 0$, we see that \eqref{eq:DecompE'} may be rewritten as
\begin{equation} \label{eq:DecompE}
	\Res{\eslvoa{k}} \cong \bigoplus_{\mu \in \rlat} \latt{\mu}{L} \otimes \pfirr{\mu}{1}
	= \bigoplus_{\lambda \in \lat{L}' / \lat{L}} \latt{\lambda}{L} \otimes \sqbrac*{\bigoplus_{\mu \in \lat{L}} \pfirr{\lambda + \mu}{1}},
\end{equation}
where $\latt{\lambda}{L}$ denotes the $\lvoa{L}$-module $\bigoplus_{\ell \in \ZZ} \fock{\lambda + 2w \ell}$ with $\lambda \in \lat{L}' = \frac{1}{v} \ZZ$.  Note that $\lat{L}'$ is the dual lattice of $\lat{L}$ with respect to the normalisation $\inner{h}{h} = 2k$ induced from the \ope{} of $h(z)$ and $h(w)$.

The coset construction applied to $\lvoa{L} \subset \eslvoa{k}$ now defines the \voa{} $\epfvoa{k} =  \Com{\lvoa{L}}{\eslvoa{k}}$ whose decomposition into $\pfvoa{k}$-modules takes the form
\begin{equation} \label{eq:DecompB}
\Res{\epfvoa{k}} \cong \bigoplus_{\mu \in \lat{L}} \pfirr{\mu}{1}.
\end{equation}
For $\mu \in \lat{L}$, the $\pfirr{\mu}{1} = \pfirr{2w \ell}{1}$ with $\ell \in \ZZ$ are simple currents (\cref{prop:pffusL}), so we conclude that $\epfvoa{k}$ is also a simple current extension of $\pfvoa{k}$.
\begin{remark}
By \cref{prop:pfconfwts}, the conformal dimension of the ground states of $\pfirr{\pm 2w}{1}$ and 
$\pfirr{\pm 4w}{1}$ are
\begin{equation}
	\pfirrdim{\pm 2w}{r}=(2v-u)(v+1) \qquad \text{and} \qquad \pfirrdim{\pm 4w}{r}=2(2v-u)(2v+1)\ge 2\pfirrdim{2w}{r}.
\end{equation}
It follows that the primary field of $\pfirr{\pm 4w}{1}$ appears in the regular terms of 
the operator product algebra obtained by extending 
the strong generators of $\pfvoa{k}$ by 
the primary fields 
of $\pfirr{\pm 2w}{1}$.  Generalising this observation, we see that this extended set 
strongly generates $\epfvoa{k}$.
\end{remark}

This development completes the picture described in the introduction and summarised in the diagram \eqref{pic:thegame}.  The aim of the rest of this section is to identify $\epfvoa{k}$-modules, compute their characters and fusion rules, and then prove the modularity of the irreducible characters: their linear span extends to a finite-dimensional representation of the modular group.

\begin{remark} \label{NotTheTriplet}
	The choice $\alpha = 2v$, leading to $\lat{L} = 2vk \ZZ = 2w \ZZ$, is not always minimal.  For instance, when $k=-\frac{4}{3}$, the constraint $k \alpha^2 \in 4 \ZZ$ is satisfied by $\alpha = v = 3$, which would lead to $\lat{L} = w \ZZ = 4 \ZZ$ instead of $8 \ZZ$.  The upshot is that the extension $\epfvoa{-4/3}$ studied here has an order two simple current of integer conformal weight.  Indeed, \cref{prop:pfconfwts} gives the conformal weight of this simple current as $\pfirrdim{4}{1} = 5$ --- the resulting simple current extension of $\epfvoa{-4/3}$ is, of course, the $c=-7$ triplet \voa{} $\tripvoa{1,3}$.
\end{remark}

Given the decomposition \eqref{eq:DecompB}, it is now straightforward to lift a $\pfvoa{k}$-module $\Mod{M}$ to a (possibly twisted) $\epfvoa{k}$-module $\Ind{\Mod{M}}$ using the fusion rules of \cref{prop:pffusL}.  Indeed,
\begin{equation}
	\Ind{\Mod{M}} = \epfvoa{k} \fus{\pfvoa{k}} \Mod{M} \qquad \implies \qquad
	\Res{(\Ind{\Mod{M}})} \cong \bigoplus_{\lambda \in \lat{L}} \pfirr{\lambda}{1} \fus{\pfvoa{k}} \Mod{M}.
\end{equation}
If $\End \Mod{M} \cong \CC$, then this lift will be an (untwisted) $\epfvoa{k}$-module if and only if it is $\ZZ$-graded (\cref{res:sclift}).  This condition is obviously satisfied for all irreducibles as well as the atypical standards.

Let us illustrate the procedure by using the fusion rule \eqref{fr:pfLE} to analyse the lift of a typical $\pfvoa{k}$-module $\pfrel{\mu}{r,s}$ with $r=1, \dots, u-1$, $s=1, \dots, v-1$ and $\mu \neq \lambda_{r,s}, \lambda_{u-r,v-s} \pmod{\rlat}$:
\begin{equation} \label{eq:defBRel}
	\epfrel{\mu}{r,s} = \epfvoa{k} \fus{\pfvoa{k}} \pfrel{\mu}{r,s}, \qquad
	\Res{\epfrel{\mu}{r,s}} \cong \bigoplus_{\lambda \in \lat{L}} \pfirr{\lambda}{1} \fus{\pfvoa{k}} \pfrel{\mu}{r,s} \cong \bigoplus_{\lambda \in \mu + \lat{L}} \pfrel{\lambda}{r,s}.
\end{equation}
\cref{prop:pfconfwts} makes it easy to check that this lift will be untwisted if and only if $\mu \in \lat{L}' = \frac{1}{v} \ZZ$, assuming that $\mu \neq \lambda_{r,s}, \lambda_{u-r,v-s} \pmod{\rlat}$.  It is also simple, by \cref{res:sclift}.  We note that $\epfrel{\lambda}{r,s}$ coincides with $\epfrel{\lambda}{u-r,v-s}$ and $\epfrel{\mu}{r,s}$ when $\lambda = \mu \pmod{\lat{L}}$.

Applying this same procedure to the atypical standard and irreducible $\pfvoa{k}$-modules gives the decompositions of the resulting $\epfvoa{k}$-modules.
\begin{proposition} \label{prop:epfdecomp}
	Given \cref{ass:cat,ass:fusion}, the typical $\pfvoa{k}$-modules $\pfrel{\mu}{r,s}$ lift to irreducible \hw{} $\epfvoa{k}$-modules, denoted by $\epfrel{\mu}{r,s}$, only if $\mu \in \lat{L}'$.  The atypical irreducible $\pfvoa{k}$-modules $\pfirr{\mu}{r}$ and $\pfdis{\mu}{r,s}$ always lift to irreducible \hw{} $\epfvoa{k}$-modules, denoted by $\epfirr{\mu}{r}$ and $\epfdis{\mu}{r,s}$, respectively.  The atypical standard $\pfvoa{k}$-modules $\pfindrel{\mu}{r,s}^{\pm}$ likewise always lift to length $2$ indecomposable $\epfvoa{k}$-modules, denoted by $\epfindrel{\mu}{r,s}^{\pm}$.  The corresponding decompositions as $\pfvoa{k}$-modules take the unified form
	\begin{equation}
		\Res{\epfmod{\mu}{\star}^{\bullet}} \cong \bigoplus_{\lambda \in \mu + \lat{L}} \pfmod{\lambda}{\star}^{\bullet},
	\end{equation}
	for appropriate $\bullet$ and $\star$, and we have $\epfmod{\lambda}{\star}^{\bullet} = \epfmod{\mu}{\star}^{\bullet}$ when $\lambda = \mu \pmod{\lat{L}}$.
\end{proposition}

\begin{remark} \label{epffinite}
	Note that the isomorphism classes of the typical $\epfvoa{k}$-modules are parametrised by $r$ and $s$, which take finitely many values, as well as (a subset of) the finite quotient $\lat{L}' / \lat{L}$.  Similarly, those of the atypical irreducibles are parametrised by $r$, $s$ and the finite quotient $\rlat / \lat{L}$.  We conclude that the $\epfvoa{k}$-module category $\epfcat{k}$ obtained from $\pfcat{k}$ by simple current extension has \emph{finitely} many simple objects, up to isomorphism.  It is not, however, clear if $\epfvoa{k}$ has finitely many irreducible modules, up to isomorphism.  Nevertheless, we are confident that this is so (see \cref{conj} below).
\end{remark}

\subsection{Fusion}

Recall that in this work we are assuming that the $\pfvoa{k}$-module category $\pfcat{k}$ can be given the structure of a vertex tensor category.  Combining \cref{res:fus} with \cref{prop:pffusL,prop:pfgrfus} therefore immediately leads to the following (Grothendieck) fusion rules.

\begin{proposition} \label{prop:epffusL}
	Given \cref{ass:cat,ass:fusion}, the fusion rules of the irreducible $\epfvoa{k}$-modules with the $\epfirr{\mu}{r}$ are
	\begin{subequations}
		\begin{align}
			\epfirr{\mu}{r} \fus{\epfvoa{k}} \epfirr{\mu'}{r'} &\cong \bigoplus_{r''=1}^{u-1} \vfusco{r,1}{r',1}{r'',1}{u,v} \epfirr{\mu+\mu'}{r''}, \label{fr:epfLL} \\
			\epfirr{\mu}{r} \fus{\epfvoa{k}} \epfdis{\mu'}{r',s'} &\cong \bigoplus_{r''=1}^{u-1} \vfusco{r,1}{r',1}{r'',1}{u,v} \epfdis{\mu+\mu'}{r'',s'}, \label{fr:epfLD} \\
			\epfirr{\mu}{r} \fus{\epfvoa{k}} \epfrel{\mu'}{r',s'} &\cong \bigoplus_{r''=1}^{u-1} \vfusco{r,1}{r',1}{r'',1}{u,v} \epfrel{\mu+\mu'}{r'',s'}. \label{fr:epfLE}
		\end{align}
	\end{subequations}
	In particular, the $\epfirr{\mu}{1}$ with $\mu \in \rlat / \lat{L}$ and the $\epfirr{\mu}{u-1}$ with $\mu \in u + \rlat / \lat{L}$ are all simple currents in the $\epfvoa{k}$-module category $\epfcat{k}$.
\end{proposition}

\begin{proposition} \label{prop:epfgrfus}
	Given \cref{ass:cat,ass:fusion}, the Grothendieck fusion rules involving the atypicals $\epfdis{\mu}{r,s}$ and the typicals $\epfrel{\mu}{r,s}$ include
	\begin{subequations}
		\begin{align}
			\Gr{\epfrel{\mu}{r,s}} \Grfus{\epfvoa{k}} \Gr{\epfrel{\mu' }{r',s'}}
			&= \sum_{(r'',s'')} \vfusco{r,s}{r',s'}{r'',s''}{u,v} \brac*{\Gr{\epfrel{\mu + \mu'-k}{r'',s''}} + \Gr{\epfrel{\mu + \mu' + k}{r'',s''}}} \notag \\
			&\mspace{-30mu} + \sum_{(r'',s'')} \brac*{\vfusco{r,s}{r',s'-1}{r'',s''}{u,v} + \vfusco{r,s}{r',s'+1}{r'',s''}{u,v}} \Gr{\epfrel{\mu + \mu'}{r'',s''}}, \\
			\Gr{\epfrel{\mu}{r,s}} \Grfus{\epfvoa{k}} \Gr{\epfdis{\mu'}{r',s'}} &= \sum_{(r'',s'')} \vfusco{r,s}{r',s'+1}{r'',s''}{u,v} \Gr{\epfrel{\mu + \mu'}{r'',s''}} \notag \\
			&\mspace{100mu} + \sum_{(r'',s'')} \vfusco{r,s}{r',s'}{r'',s''}{u,v} \Gr{\epfrel{\mu+\mu'-k}{r'',s''}}.
			\intertext{If $s+s'<v$, then we have in addition}
			\Gr{\epfdis{\mu}{r,s}} \Grfus{\epfvoa{k}} \Gr{\epfdis{\mu'}{r',s'}} &= \sum_{(r'',s'')} \vfusco{r,s}{r',s'}{r'',s''}{u,v} \Gr{\epfrel{\mu+\mu'-k}{r'',s''}} \notag \\
			&\mspace{100mu}+ \sum_{r''=1}^{u-1} \vfusco{r,1}{r',1}{r'',1}{u, v} \Gr{\epfdis{\mu+\mu'}{r'',s+s'}},
			\intertext{while if $s+s' \ge v$, then we have instead}
			\Gr{\epfdis{\mu}{r,s}} \Grfus{\epfvoa{k}} \Gr{\epfdis{\mu'}{r',s'}} &= \sum_{(r'',s'')} \vfusco{r,s+1}{r',s'+1}{r'',s''}{u,v} \Gr{\epfrel{\mu+\mu'-k}{r'',s''}} \notag \\
			&\mspace{100mu} + \sum_{r''=1}^{u-1} \vfusco{r,1}{r',1}{r'',1}{u,v} \Gr{\epfdis{\mu+\mu'-k}{u-r'',s+s'-v+1}}.
		\end{align}
	\end{subequations}
\end{proposition}

\subsection{Examples} \label{epfexamples}

As noted in \cref{pfexamples}, the parafermion coset $\pfvoa{-1/2}$ is isomorphic to the singlet \voa{} $\singvoa{1,2}$.  It is therefore not surprising that the extension $\epfvoa{-1/2}$ is isomorphic to the triplet $\tripvoa{1,2}$.  This is consistent with $\lat{L} = 2\ZZ$ and $\pfirrdim{\pm 2}{1} = 3$; for more detail, see \cite{CreCos13}.  We also recalled that $\pfvoa{-4/3}$ is isomorphic to the singlet $\singvoa{1,3}$.  However, $\epfvoa{-4/3}$ is not $\tripvoa{1,3}$: as discussed in \cref{NotTheTriplet}, it is rather a $\ZZ_2$-orbifold of $\tripvoa{1,3}$.

Consider the extended parafermion coset $\epfvoa{-2/3}$, remembering that $\pfvoa{-2/3}$ has been identified with the bosonic orbifold of the supersinglet \svoa{} $\ssingvoa{1,3}$.  Since $\lat{L} = 4\ZZ$ and $\pfirrdim{\pm 4}{1} = 9$, $\epfvoa{-2/3}$ is not the bosonic orbifold of the supertriplet \svoa{} $\stripvoa{1,3}$ \cite{AdaN=109}.  Indeed, $\epfvoa{-2/3}$ has three simple currents $\epfirr{0}{3}$, $\epfirr{2}{1}$ and $\epfirr{2}{3}$ whose spaces of ground states have dimensions and conformal weights $1$ and $\frac{3}{2}$, $2$ and $\frac{5}{2}$, and $2$ and $3$, respectively.  Under fusion, they form (along with the vacuum module $\epfirr{0}{1}$) a group isomorphic to $\ZZ_2 \oplus \ZZ_2$:
\begin{equation}
	\epfirr{\mu}{r} \fus{\epfvoa{-2/3}} \epfirr{\mu'}{r'} = \epfirr{\mu+\mu'}{r+r'-1 \mspace{-15mu} \pmod{4}} \qquad \text{(\(\mu, \mu' \in \set{0,2}\), \(r,r' \in \set{1,3}\)).}
\end{equation}
The simple current extension of $\epfvoa{-2/3}$ by this group of simple currents is a \svoa{} with strong generating fields of conformal dimensions $\frac{3}{2}$, $2$, $\frac{5}{2}$, $\frac{5}{2}$, $\frac{5}{2}$, $3$, $3$ and $3$. This \svoa{} and the supertriplet $\stripvoa{1,3}$ are thus extensions of the same \svoa{}  $\ssingvoa{1,3}$ with same type of strong generators. We expect that one can now prove that they have to coincide by using the Jacobi identity to show that $\ssingvoa{1,3}$ admits at most one such extension. We omit this long computation and refer to the proofs of \cite[Thm.~3.1]{AraMin17} and \cite[Lem.~8.2]{AraBP15} for similar arguments. 

We shall instead demonstrate this coincidence of \svoas{} by proving that 
the simple current extension of $\epfvoa{-2/3}$ is a 
subalgebra of $\stripvoa{1,3}$. Since both \voas{} have the same type of minimal strong generating set, they 
must therefore coincide. We noted in \cref{pfexamples} 
that $\epfirr{0}{3}$ is contained in $\ssingvoa{1,3}$. 
Hence, we only need 
find one other 
generator of the group of simple currents inside 
$\stripvoa{1,3}$. 
We will now show that indeed $\epfirr{2}{1}$ is contained in $\stripvoa{1,3}$.

For this, we use the notation of \cite{AdaN=109}, while the spirit of the proof is closer to the arguments in \cite{CreCos13}. Let $\VOA{F}$ be the vertex superalgebra of a single free fermion. Consider the lattice $\alpha\ZZ\oplus\beta\ZZ$, with $\alpha^2=-\beta^2=3$ and $\alpha\beta=0$. The $\stripvoa{1,3}$-algebra is defined \cite{AdaN=109} as the kernel of a screening operator $Q$ acting on 
the tensor product of $\VOA{F}$ and the lattice \voa{} $\lvoa{\alpha\ZZ}$:
\begin{equation} \label{eq:DefSW13}
	\stripvoa{1,3} = \ker_Q(\lvoa{\alpha\ZZ} \otimes \VOA{F}) = \bigoplus_{n \in \ZZ} \ker_Q(\fock{n\alpha} \otimes \VOA{F}).
\end{equation}
In particular, it is a 
subalgebra of 
$\lvoa{\alpha\ZZ \oplus \beta\ZZ} \otimes \VOA{F}$.

We denote the vertex operator corresponding to the highest-weight vector of the Fock module $\fock{n\alpha+m\beta}$ by $\ee^{n\alpha+m\beta}$. The odd triplet fields of conformal weight $\frac{5}{2}$ are denoted by $E$, $H$ and $F$ in \cite{AdaN=109} and they belong to 
$\fock{\alpha} \otimes \VOA{F}$, $\fock{0} \otimes \VOA{F}$ and $\fock{-\alpha} \otimes \VOA{F}$, respectively.  Now, $\stripvoa{1,3}$ is simple
\cite[Cor.~10.1]{AdaN=109}, so it admits
a non-degenerate invariant bilinear form \cite{Li} and thus 
every field of conformal weight $h$ must possess a conjugate
field of same conformal weight such that their operator product expansion involves the identity field. 
For the super-triplet field $E$, the only possibility for the conjugate field is $F$, 
so we have
\begin{equation}
	E(z)F(w) \sim \frac{a}{(z-w)^5} +\frac{0}{(z-w)^4} +\cdots,
\end{equation}
for some non-zero $a$. 
We 
rescale 
$E$ and/or $F$ so
that $a=-\frac{2}{3}$.

Define now the fields $e(z)=E(z)e^{\beta} \in \fock{\alpha+\beta} \otimes \VOA{F}$ and $f(z)=F(z)e^{-\beta} \in \fock{-\alpha-\beta} \otimes \VOA{F}$ whose 
conformal weights are both $\frac{5}{2}-\frac{3}{2}=1$. Let $h(z)$ be the Heisenberg field of $\lvoa{\ZZ\beta}$, normalised such that 
\begin{equation}
h(z)e(w) \sim \frac{2e(w)}{z-w}, \qquad h(z)f(w) \sim \frac{-2f(w)}{z-w}\qquad \text{so that} \qquad h(z)h(w)\sim \frac{-4/3}{(z-w)^2}.
\end{equation}
Then, the
fields $e$, $h$ and $f$ generate an
affine \voa{} $\VOA{X}$ with $\alg{g} = \SLA{sl}{2}$ and level $k = -\frac{2}{3}$. 
Suppose that $\VOA{X}$ is not
simple, hence that it has
a non-trivial proper ideal $\VOA{I}$. Let $v$ be a vector of minimal 
conformal weight in
$\VOA{I}$. Applying the zero-modes of $e$ or $f$ if necessary, we may
assume that $v$ has $\SLA{sl}{2}$-weight zero and hence belongs to
the commutant of the Heisenberg field $h$ in
$\VOA{X}$.  However, we identified this commutant as
the simple supersinglet algebra $\ssingvoa{1,3}$ in \cref{pfexamples}, 
which obviously
has no non-trivial proper ideals.  This contradiction proves that $\VOA{X}$ is simple, hence
that $\VOA{X}\cong \slvoa{-2/3}$.

We thus have the following inclusion of \voas:
\begin{equation}
\slvoa{-2/3} \subset \ker_Q(\lvoa{(\alpha+\beta)\ZZ}\otimes \VOA{F}). 
\end{equation}
Here, we have noted that $Q$ annihilates both $e$ and $f$.  Decomposing the lattice \voa{} into Fock spaces now gives 
\begin{equation}
\ker_Q(\lvoa{(\alpha+\beta)\ZZ}\otimes \VOA{F}) \cong \bigoplus_{n\in\ZZ} \ker_Q(\fock{n\alpha+n\beta}\otimes \VOA{F}) \cong \bigoplus_{n\in\ZZ} \ker_Q(\fock{n\alpha}\otimes \VOA{F}) \otimes \fock{2n},
\end{equation}
whilst the decomposition of the affine \voa{} into $\VOA{H} \otimes \pfvoa{-2/3}$-modules is
\begin{equation}
\slvoa{-2/3} \cong \bigoplus_{n\in\ZZ} \pfirr{2n}{1} \otimes \fock{2n}.
\end{equation}
It follows that $\pfirr{2n}{1} \subset \ker_Q(\fock{n\alpha}\otimes \VOA{F})$, for all $n \in \ZZ$,   
and so
the simple current $\epfirr{2}{1}$ satisfies
\begin{equation}
\epfirr{2}{1} = \bigoplus_{n\in 2\ZZ+1} \pfirr{2n}{1} \subset \bigoplus_{n\in\ZZ} \pfirr{2n}{1} \subset \bigoplus_{n\in\ZZ} \ker_Q(\fock{n\alpha}\otimes \VOA{F}) = \stripvoa{1,3},
\end{equation}
by \eqref{eq:DefSW13}.  This proves that the simple current extension of $\epfvoa{-2/3}$ introduced above is isomorphic to $\stripvoa{1,3}$, as claimed.

We conclude by noting that the spectrum of irreducible $\epfvoa{-2/3}$-modules comprises $6$ of $\slirr{}$-type, $6$ (inequivalent) of $\sldis{}$-type and $24$ of $\slindrel{}$-type.  Summing over the orbits of the group of simple currents, we arrive at $2$, $2$ and $8$ inequivalent $\stripvoa{1,3}$-modules (not accounting for global parities) whose properties are summarised in the following table: \\[1.5ex]
{\centering	\scalebox{0.82}{
	\begin{tabular}{c|CC|CC|CCCCCCCC}
		& \pfirr{0}{1} & \pfirr{1}{2} & \pfdis{2/3}{1,1} & \pfdis{5/3}{2,1} & \pfrel{0}{1;1} & \pfrel{1/3}{1;1} & \pfrel{1}{1;1} & \pfrel{5/3}{1,1} & \pfrel{0}{2,1} & \pfrel{2/3}{2,1} & \pfrel{1}{2,1} & \pfrel{4/3}{2,1} \\
		$\Delta$ & 0 & \frac{15}{16} & \frac{1}{2} & -\frac{1}{16} & -\frac{1}{6} & -\frac{1}{8} & \frac{5}{24} & -\frac{1}{8} & -\frac{5}{48} & \frac{1}{16} & \frac{13}{48} & \frac{1}{16} \\
		Sector & (NS,NS) & (R,R) & (NS,NS) & (R,R) & (NS,NS) & (NS,R) & (R,R) & (NS,R) & (R,NS) & (R,NS) & (R,R) & (R,NS)
	\end{tabular}
}} \\[1.5ex]
Here, $\Delta$ denotes the conformal weight of the ground states and ``Sector'' gives the $N=1$ and super-$\VOA{W}_3$ sectors as an ordered pair.  

\subsection{Standard $\epfvoa{k}$-characters and modularity}

Given the decompositions \eqref{eq:defBRel}, we can clearly sum the characters \eqref{eq:pfrelchar} of the standard $\pfvoa{k}$-modules to obtain those of the standard $\epfvoa{k}$-modules.
\begin{proposition} \label{prop:epfstchars}
	Given \cref{ass:cat,ass:fusion}, the characters of the standard $\epfvoa{k}$-modules are
	\begin{equation} \label{eq:epfstchar}
		\fch{\epfrel{\mu}{r,s}}{q} = \fch{\epfindrel{\mu}{r,s}^{\pm}}{q} = \frac{\vch{u,v}{r,s}{q}}{\eta(q)} \sum_{\lambda \in \mu + \lat{L}} q^{-\lambda^2 / 4k}.
	\end{equation}
\end{proposition}

To investigate the modularity of these characters, we introduce the following theta functions associated to the lattice $\lat{L}$:
\begin{equation} \label{eq:DefTheta}
	\fjth{\mu + \lat{L}}{z,q} = \sum_{\lambda \in \mu + \lat{L}} z^{\lambda} q^{-\lambda^2 / 4k} \qquad \text{(\(\mu \in \lat{L}'\)).}
\end{equation}
Writing $z = \ee^{2 \pi \ii \zeta}$ and $q = \ee^{2 \pi \ii \tau}$ as usual, the modular S-transforms of these theta functions are given by
\begin{equation} \label{eq:ThetaS}
	\fjth{\mu + \lat{L}}{\zeta / \tau \mid -1/\tau} = \frac{\sqrt{-\ii \tau} \, \ee^{-2 \pi \ii k \zeta^2 / \tau}}{\sqrt{\abs*{\lat{L}' / \lat{L}}}} \sum_{\lambda \in \lat{L}' / \lat{L}} \ee^{\ii \pi \lambda \mu / k} \fjth{\lambda + \lat{L}}{\zeta \mid \tau}.
\end{equation}

To compare with characters, we shall need to set $\zeta = 0$ ($z=1$) and will then drop $\zeta$ (or $z$) from the list of arguments.  With this convention, it is obvious that the theta functions are invariant under reflection about zero, $\fjth{-\mu + \lat{L}}{\tau} = \fjth{\mu + \lat{L}}{\tau}$, and translations in $\lat{L}$.  We shall refer to these properties as the affine Weyl symmetry of the theta functions.  Moreover, if we set $p = \frac{1}{2} \abs*{\lat{L}' / \lat{L}} = vw$, then \eqref{eq:ThetaS} becomes
\begin{equation} \label{eq:ThetaS'}
	\fjth{m/v + \lat{L}}{-1/\tau} = \frac{\sqrt{-\ii \tau}}{\sqrt{2p}} \sum_{\ell=0}^{2p-1} \ee^{-\ii \pi \ell m/p} \fjth{\ell/v + \lat{L}}{\tau} = \sqrt{-\ii \tau} \sum_{\ell=0}^p \Styp{m \ell} \fjth{\ell/v + \lat{L}}{\tau},
\end{equation}
where
\begin{equation} \label{eq:DefStyp}
	\Styp{m \ell} =
	\begin{cases*}
		\sqrt{\frac{1}{2p}} \cos \frac{\pi \ell m}{p} & if \(\ell \in p \ZZ\), \\
		\sqrt{\frac{2}{p}} \cos \frac{\pi \ell m}{p} & otherwise
	\end{cases*}
\end{equation}
and we may restrict $m$ to $0,1,\dots,p$.

Since the standard $\epfvoa{k}$-characters \eqref{eq:epfstchar} can be written in the form
\begin{equation}
	\fch{\epfrel{\mu}{r,s}}{\tau} = \frac{\vch{u,v}{r,s}{\tau}}{\eta(\tau)} \fjth{\mu + \lat{L}}{\tau},
\end{equation}
it is now easy to obtain their modular S-transforms.  We recall that the S-matrix of the Virasoro minimal model $\virminmod{u,v}$ is
\begin{equation} \label{eq:DefSvir}
	\Svir{u,v}{r,s}{r',s'} = -2 \sqrt{\frac{2}{uv}} \brac{-1}^{rs'+r's} \sin \frac{v \pi r r'}{u} \sin \frac{u \pi s s'}{v},
\end{equation}
where the entries $(r,s)$ and $(r',s')$ run over the irreducible $\virminmod{u,v}$-modules of the Kac table $\kactable{u,v}$, see \eqref{eq:DefKacTable}.  As before, sums indexed by $\virminmod{u,v}$-modules will always be assumed to run over $\kactable{u,v}$.
\begin{proposition} \label{prop:epfstcharmod}
	Given \cref{ass:cat,ass:fusion}, the modular S-transforms of the standard $\epfvoa{k}$-characters are
	\begin{equation} \label{eq:epfstcharmod}
		\fch{\epfrel{m/v}{r,s}}{-1/\tau} = \sum_{(r',s')} \sum_{\ell=0}^p \Svir{u,v}{r,s}{r',s'} \Styp{m \ell} \fch{\epfrel{\ell/v}{r',s'}}{\tau}.
	\end{equation}
\end{proposition}

\begin{remark}
	Note that the S-matrix appearing in \eqref{eq:ThetaS} is symmetric whilst that of \eqref{eq:DefStyp} is not.  This is not unexpected because the standard $\epfvoa{k}$-characters, as we have defined them, are not linearly independent.
\end{remark}

\subsection{Atypical $\epfvoa{k}$-characters}

In principle, the decompositions of \cref{prop:epfdecomp} also yield character formulae for the irreducible atypical $\epfvoa{k}$-modules $\epfirr{\mu}{r}$ and $\epfdis{\mu}{r,s}$.  For example, substituting the resolution formulae from \eqref{eq:pfatypchar} results in the following expression for the former:
\begin{equation} \label{eq:epfirrchar1}
	\ch{\epfirr{\mu}{r}} = \sum_{s=1}^{v-1} (-1)^{s-1} \frac{\vch{u,v}{r,s}{q}}{\eta(q)} \sum_{\lambda \in \mu + \lat{L}} \sum_{m=0}^{\infty} \brac*{q^{-(\lambda + 2mw-sk)^2 / 4k} - q^{-(\lambda + 2(m+1)w+sk)^2 / 4k}}.
\end{equation}
Consider the double sum in this expression, rewritten in the form
\begin{equation}
	\sum_{\ell \in \ZZ} \sum_{m=0}^{\infty} \brac*{q^{-(\mu-sk + 2w(\ell+m))^2 / 4k} - q^{-(\mu+sk + 2w(\ell+m+1))^2 / 4k}}.
\end{equation}
This sum is clearly not absolutely convergent, so we must take care in how we manipulate its terms.

Note first that the \lhs{} of \eqref{eq:epfirrchar1} is invariant under $\mu \mapsto -\mu$, by \cref{rem:pfirraltchar,prop:epfdecomp}.  The same must therefore be true for the \rhs{}.  Replacing $\ch{\epfirr{\mu}{r}}$ by $\frac{1}{2} (\ch{\epfirr{\mu}{r}} + \ch{\epfirr{-\mu}{r}})$ transforms the double sum into $\frac{1}{2} \brac*{A_{\mu+sk}(q) + A_{-\mu+sk}(q)}$, where
\begin{equation} \label{eq:DefA}
	A_{\lambda}(q) = \sum_{\ell \in \ZZ} \sum_{m=0}^{\infty} \brac*{q^{-(\lambda - 2w(\ell+m))^2 / 4k} - q^{-(\lambda + 2w(\ell+m+1))^2 / 4k}}.
\end{equation}
This will be identified (see \cref{partialTheta} below) with a linear combination of the theta functions $\jth{\mu + \lat{L}}$, $\mu \in \lat{L}'$, of \eqref{eq:DefTheta} and their derivatives
\begin{equation} \label{eq:DefTheta'}
	\fdjth{\mu + \lat{L}}{z,q} = -\frac{z \partial_z}{2w} \fjth{\mu + \lat{L}}{z,q}
	= -\frac{\mu}{2w} \fjth{\mu + \lat{L}}{z,q} + \sum_{\ell \in \ZZ} \ell z^{\mu - 2w \ell} q^{-(\mu - 2w \ell)^2 / 4k}.
\end{equation}
As usual, we may omit $z$ (or $\zeta$) from the argument of these theta functions, understanding that it is then evaluated at $z=1$ (or $\zeta = 0$).  These specialised derivatives are affine Weyl-antisymmetric, being $\lat{L}$-periodic and anti-invariant under reflection: $\fdjth{-\mu + \lat{L}}{q} = -\fdjth{\mu + \lat{L}}{q}$.

We record the following easily proven identities for the \lcnamecref{partialTheta} that follows:
\begin{subequations}
	\begin{align}
		\fjth{\mu + \lat{L}}{q} &= \sum_{\ell=0}^{\infty} \brac*{q^{-(\mu - 2w \ell)^2 / 4k} + q^{-(\mu + 2w (\ell+1))^2 / 4k}}, \label{eq:ThetaIdentity} \\
		\fdjth{\mu + \lat{L}}{q} &= -\frac{\mu}{2w} \fjth{\mu + \lat{L}}{q} + \sum_{\ell=0}^{\infty} (\ell+1) \brac*{q^{-(\mu - 2w \ell)^2 / 4k} - q^{-(\mu + 2w (\ell+1))^2 / 4k}} - \sum_{\ell=0}^{\infty} q^{-(\mu - 2w \ell)^2 / 4k}. \label{eq:Theta'Identity}
	\end{align}
\end{subequations}

\begin{lemma} \label{partialTheta}
For any $\lambda \in \lat{L}'$, we have
\begin{equation}
	A_{\lambda}(q) = 2 \fdjth{\lambda + \lat{L}}{q} + \brac*{1 + \frac{\lambda}{w}} \fjth{\lambda + \lat{L}}{q}.
\end{equation}
\end{lemma}
\begin{proof}
Consider first the partial sum
\begin{equation}
	A_{\lambda}^+(q) = \sum_{\ell=0}^{\infty} \sum_{m=0}^{\infty} \brac*{q^{-(\lambda - 2w(\ell+m))^2 / 4k} - q^{-(\lambda + 2w(\ell+m+1))^2 / 4k}}.
\end{equation}
Replacing $m$ by $n = \ell + m$ and swapping the order of summation gives
\begin{align}
	A_{\lambda}^+(q) &= \sum_{n=0}^{\infty} (n+1) \brac*{q^{-(\lambda - 2wn)^2 / 4k} - q^{-(\lambda + 2w(n+1))^2 / 4k}} \notag \\
	&= \fdjth{\lambda + \lat{L}}{q} + \frac{\lambda}{2w} \fjth{\lambda + \lat{L}}{q} + \sum_{n=0}^{\infty} q^{-(\lambda - 2wn)^2 / 4k},
\end{align}
by \eqref{eq:Theta'Identity}.

Sending $\ell$ to $-\ell-1$ and setting $m=n+2\ell+1$ in the complementary partial sum
\begin{equation}
	A_{\lambda}^-(q) = \sum_{\ell=-\infty}^{-1} \sum_{m=0}^{\infty} \brac*{q^{-(\lambda - 2w(\ell+m))^2 / 4k} - q^{-(\lambda + 2w(\ell+m+1))^2 / 4k}},
\end{equation}
we instead arrive at
\begin{equation}
	A_{\lambda}^-(q) = A_{\lambda}^+(q) + \sum_{\ell=0}^{\infty} \sum_{n=-2\ell-1}^{-1} \brac*{q^{-(\lambda - 2w(\ell+n))^2 / 4k} - q^{-(\lambda + 2w(\ell+n+1))^2 / 4k}}.
\end{equation}
Noting that for fixed $\ell$, the summand indexed by $n$ precisely cancels that indexed by $-2\ell-1-n$, we see that only the $n=-2\ell-1$ summand contributes.  We conclude that
\begin{equation}
	A_{\lambda}^-(q) = A_{\lambda}^+(q) + \fjth{\lambda + \lat{L}}{q} - 2 \sum_{\ell=0}^{\infty} q^{-(\lambda - 2w \ell)^2 / 4k},
\end{equation}
by \eqref{eq:ThetaIdentity}.  Since $A_{\lambda}(q) = A_{\lambda}^+(q) + A_{\lambda}^-(q)$, the proof is complete.
\end{proof}

Recalling that $\jth{\mu + \lat{L}}$ and $\djth{\mu + \lat{L}}$ are Weyl-symmetric and Weyl-antisymmetric, respectively, we can use this result to express the character formula \eqref{eq:epfirrchar1} as follows.
\begin{proposition} \label{prop:epfirrchar=theta}
	Given \cref{ass:cat,ass:fusion}, we have
	\begin{align} \label{eq:epfirrchar=theta}
		\fch{\epfirr{\mu}{r}}{q} &= \sum_{s=1}^{v-1} (-1)^{s-1} \frac{\vch{u,v}{r,s}{q}}{\eta(q)} \biggl[ \fdjth{\mu + sk + \lat{L}}{q} - \fdjth{\mu - sk + \lat{L}}{q} \biggr. \notag \\
		&\mspace{100mu} \biggl. + \frac{\mu - (v-s)k}{2w} \fjth{\mu + sk + \lat{L}}{q} - \frac{\mu + (v-s)k}{2w} \fjth{\mu - sk + \lat{L}}{q} \biggr],
	\end{align}
	for all $r=1,\dots,u-1$ and $\mu = \lambda_{r,0} = r-1 \pmod{\rlat}$.
\end{proposition}

\begin{remark} \label{rem:D=E+L}
	\cref{eq:pfdischar,prop:epfdecomp} imply that the characters of the remaining atypical irreducible $\epfvoa{k}$-modules are given by
	\begin{equation}
		\fch{\epfdis{\mu}{r,s}}{q} = \sum_{s'=s+1}^{v-1} (-1)^{s'-s-1} \fch{\epfrel{\mu-(s'-s)k}{r,s'}}{q} + (-1)^{v-1-s} \fch{\epfirr{\mu-(v-s)k}{u-r}}{q}.
	\end{equation}
	We shall not try to simplify this expression.  As the standard $\epfvoa{k}$-characters close on themselves under modular transformations, those of the $\sldis{}$-type $\epfvoa{k}$-characters will transform as a linear combination of standard and $\slirr{}$-type $\epfvoa{k}$-characters if the $\slirr{}$-type $\epfvoa{k}$-characters do.  We therefore only need to demonstrate this result for $\slirr{}$-type $\epfvoa{k}$-characters.
\end{remark}

\subsection{Interlude: linear dependences}

Now that we have the characters of the $\epfirr{\mu}{r}$ in terms of theta functions and their derivatives, it is in principle straightforward to determine their modular S-transforms.  To this end, we shall analyse certain linear dependences that arise in the terms that appear in these characters.

We first note that we may restrict attention to the terms that involve the theta function derivatives in the character formula \eqref{eq:epfirrchar=theta} --- the other terms may be expressed in terms of standard $\epfvoa{k}$-characters and these S-transform into one another, by \cref{prop:epfstcharmod}.  We isolate these terms in the following definition:
\begin{equation} \label{eq:DefGamma}
	\fepfone{\mu}{r}{q} = \sum_{s=1}^{v-1} (-1)^{s-1} \frac{\vch{u,v}{r,s}{q}}{\eta(q)} \sqbrac*{\fdjth{\mu + sk + \lat{L}}{q} - \fdjth{\mu - sk + \lat{L}}{q}}, \qquad \mu = r-1 \pmod{\rlat}.
\end{equation}
Note that these terms constitute the part of the atypical module characters of modular weight $1$.  The remaining part has modular weight $0$ so there can be no (non-trivial) linear dependences between the parts.

\begin{lemma} \label{ObviousRedundancies}
Given $r=1,\dots,u-1$ and $\mu = r-1 \pmod{\rlat}$, we have
\begin{equation} \label{eq:Wt1Symmetries}
	\epfone{\mu}{r} = \epfone{\mu+2w}{r}, \quad
	\epfone{\mu}{r} = \epfone{-\mu}{r}, \quad
	\epfone{\mu}{r} = (-1)^{v-1} \epfone{w+\mu}{u-r}, \quad
	\epfone{\mu}{r} = (-1)^{v-1} \epfone{w-\mu}{u-r},
\end{equation}
where we recall that $w=-vk=2v-u$.
\end{lemma}
\begin{proof}
The first two identities reflect the affine Weyl-antisymmetry of the theta function derivatives.  The third uses, in addition, the Kac symmetry of the Virasoro minimal model characters:
\begin{align}
	\fepfone{\mu}{r}{q} &= \sum_{v=1}^{s-1} (-1)^{s-1} \frac{\vch{u,v}{r,s}{q}}{\eta(q)} \sqbrac*{\fdjth{\mu + sk + \lat{L}}{q} - \fdjth{\mu - sk + \lat{L}}{q}} \notag \\
	&= \sum_{v=1}^{s-1} (-1)^{v-s-1} \frac{\vch{u,v}{r,v-s}{q}}{\eta(q)} \sqbrac*{\fdjth{\mu + (v-s)k + \lat{L}}{q} - \fdjth{\mu - (v-s)k + \lat{L}}{q}} \notag \\
	&= (-1)^{v-1} \sum_{v=1}^{s-1} (-1)^{s-1} \frac{\vch{u,v}{u-r,s}{q}}{\eta(q)} \sqbrac*{\fdjth{\mu + w+sk + \lat{L}}{q} - \fdjth{\mu + w-sk + \lat{L}}{q}}.
\end{align}
Note that $\mu = r-1 \pmod{\rlat}$ implies that $\mu + w = r-1+2v-u = u-r-1 \pmod{\rlat}$, as required.  The fourth is obtained by combining the second and third.
\end{proof}

The first two relations of \eqref{eq:Wt1Symmetries} allow us to restrict $\mu$ to a fundamental domain of the affine Weyl group $\ZZ_2 \ltimes 2w \ZZ$.  We choose $0 \le \mu \le w$, remembering that $\mu$ must match $r-1$ in parity.  Next, the fourth relation states that we can always exchange $r$ for $u-r$, hence we may impose $r \le \frac{u}{2}$.  If $u$ is odd, then we are done.  If $u$ is even, however, then we can slightly refine the analysis by noting that when $r = \frac{u}{2}$, the fourth relation lets us exchange $\mu$ for $w-\mu$, hence we may insist that $\mu \le \frac{w}{2}$ in this case (note that $u$ even implies that $w$ is also even).

Consider the vector space $\vvmf{k}$ spanned by the $\epfone{\mu}{r}$, with $r=1,\dots,u-1$ and $\mu = r-1 \pmod{\rlat}$.  The assertions above allow us to significantly reduce this spanning set.
\begin{proposition} \label{NotABasis}
	A spanning set of $\vvmf{k}$ is given by the elements $\epfone{\mu}{r}$ with
	\begin{center}
		\setlength{\tabcolsep}{1em}
		\renewcommand{\arraystretch}{1.3}
		\begin{tabular}{>{\bfseries}rLLl}
			$u$ odd: & \mu = 0, 1, \dots, w, & r = 1, 2, \dots, \tfrac{1}{2} (u-1), \\
			\multirow{2}{*}{$u$ even:} & \mu = 0, 1, \dots, w, & r = 1, 2, \dots, \tfrac{1}{2} u - 1 & and \\
			                           & \mu = 0, 1, \dots, \frac{w}{2}, & r = \tfrac{1}{2} u,
		\end{tabular}
	\end{center}
	subject to $\mu = r-1 \pmod{\rlat}$.
\end{proposition}
\noindent We let $\mathbf{B}_k$ denote the set of pairs $(\mu; r)$ satisfying these conditions.
\begin{cor} \label{DimVVMF}
  \(\dim_\CC \vvmf{k} \le \abs*{\mathbf{B}_k} =
  \begin{cases*}
		\frac{1}{4} (u-1)(w+1) & if $u$ is odd, \\
		\frac{1}{4} uw - \frac{1}{2} (v-1-u) & if $u$ is even.
	\end{cases*}
	\)
\end{cor}

\begin{remark}
	We believe that the spanning set $\set{\epfone{\mu}{r} \st (\mu; r) \in \mathbf{B}_k}$ of \cref{NotABasis} is actually a basis, hence that the upper bound of \cref{DimVVMF} is actually an equality.  However, we have not tried to prove this as it is not needed for the modularity result that follows.
\end{remark}

\subsection{Modularity of atypical $\epfvoa{k}$-characters}

We turn now to the modular properties of the weight $1$ terms $\epfone{\mu}{r}$ of the characters of the $\epfirr{\mu}{r}$, first determining the S-transform of the theta function derivatives
\begin{equation}
	\fdjth{\mu + \lat{L}}{\zeta \mid \tau} = -\frac{\partial_{\zeta}}{2 \pi \ii \cdot 2w} \fjth{\mu + \lat{L}}{\zeta \mid \tau}.
\end{equation}
This is easily derived by differentiating \cref{eq:ThetaS} with respect to $\zeta$, resulting in
\begin{equation} \label{eq:Theta'S}
	\fdjth{\mu + \lat{L}}{\zeta/\tau \mid -1/\tau} = \frac{\sqrt{-\ii \tau}}{\sqrt{\abs*{\lat{L}' / \lat{L}}}} \ee^{-2 \pi \ii k \zeta^2 / \tau} \sum_{\lambda \in \lat{L}' / \lat{L}} \ee^{\ii \pi \lambda \mu / k} \sqbrac*{\tau \fdjth{\lambda + \lat{L}}{\zeta \mid \tau} - \frac{\zeta}{v} \fjth{\lambda + \lat{L}}{\zeta \mid \tau}}.
\end{equation}
Specialising to $\zeta = 0$ and $\mu = m/v$ therefore gives
\begin{equation} \label{eq:Theta'S'}
	\fdjth{m/v + \lat{L}}{-1/\tau} = \frac{\tau \sqrt{-\ii \tau}}{\sqrt{2p}} \sum_{\ell=0}^{2p-1} \ee^{-\ii \pi \ell m/p} \fdjth{\ell/v + \lat{L}}{\tau} = (-\ii \tau)^{3/2} \sum_{\ell=1}^{p-1} \Stheta{m \ell} \fdjth{\ell/v + \lat{L}}{\tau},
\end{equation}
where we recall that $p = vw = v(2v-u)$ and have set
\begin{equation} \label{eq:DefSone}
	\Stheta{m \ell} = \sqrt{\frac{2}{p}} \sin \frac{\pi \ell m}{p}.
\end{equation}
Because $\fdjth{\ell/v + \lat{L}}{\tau} = 0$, for $\ell = 0$ or $p$, it is natural to restrict the range of $\ell$ to $1, \dots, p-1$.  However, the S-matrix entries $\Stheta{m \ell}$ are $0$ for both $\ell = 0$ and $\ell = p$, so these values may be included in the summation range when convenient.

\begin{theorem} \label{thm:atypicalmodularity}
	Given \cref{ass:cat,ass:fusion}, the elements of $\vvmf{k}$ define a finite-dimensional vector-valued modular form with
	\begin{subequations}
		\begin{equation}
			\fepfone{\mu}{r}{-1/\tau} = -\ii \tau \sum_{(\mu'; r') \in \mathbf{B}_k} \Sone{(\mu; r) (\mu'; r')} \fepfone{\mu'}{r'}{\tau}, \qquad
			\fepfone{\mu}{r}{\tau+1} = \ee^{2 \pi \ii (\pfirrdim{\mu}{r} - \tilde{c} / 24)} \fepfone{\mu}{r}{\tau},
		\end{equation}
		where $\pfirrdim{\mu}{r}$ was given in \eqref{eq:pfirrconfwts}, $\tilde{c} = 2 - \frac{6}{t}$, $\mathbf{B}_k$ was given in \cref{NotABasis} and
		\begin{equation}
			\Sone{(\mu; r) (\mu'; r')} = \frac{2A_{\mu'; r'}}{\sqrt{uw}} \sin \frac{\pi r r'}{t} \cos \frac{\pi \mu \mu'}{k}, \qquad
			A_{\mu'; r'} =
			\begin{cases*}
				\frac{1}{2} & if $r' = \frac{u}{2}$ and $\mu' \in w \ZZ$, \\
				2 & if $r' \neq \frac{u}{2}$ and $\mu' \notin w \ZZ$, \\
				1 & otherwise.
			\end{cases*}
		\end{equation}
	\end{subequations}
\end{theorem}
\begin{proof}
	The behaviour under the modular T-transform $\tau \mapsto \tau + 1$ follows immediately from the relation between $\epfone{\mu}{r}$ and the character of $\epfirr{\mu}{r}$, noting that the conformal weights of the latter match those of $\pfirr{\mu}{r}$ (up to an integer) which were given in \cref{prop:pfconfwts}.  We therefore turn to the S-transform.

	From the definition \eqref{eq:DefGamma} of the weight $1$ parts of the atypical characters, we deduce that
	\begin{equation}
		\fepfone{\mu}{r}{-1/\tau} = -\ii \tau \sum_{s=1}^{v-1} (-1)^{s-1} \frac{\sum_{(r',s')} \Svir{u,v}{r,s}{r',s'} \vch{u,v}{r',s'}{\tau}}{\eta(\tau)} \cdot \frac{-2}{\sqrt{2p}} \sum_{\ell=0}^{2p-1} \ee^{-\ii \pi \ell \mu / w} \sin \frac{\pi \ell s}{v} \fdjth{\ell / v + \lat{L}}{\tau},
	\end{equation}
	where the modular S-matrix of $\virminmod{u,v}$ was given in \eqref{eq:DefSvir} and we have used the first equality of \eqref{eq:Theta'S'}.  The sum over $s$ is easily evaluated:
	\begin{align} \label{eq:TheSSum}
		\sum_{s=1}^{v-1} (-1)^{s-1} (-1)^{r's} \sin(\pi ts' s) \sin \frac{\pi \ell s}{v}
		&= \sum_{s=1}^{v-1} \cos(\pi (r'-1) s) \sin(-\pi ts' s) \sin \frac{\pi \ell s}{v} \notag \\
		&= \sum_{s=1}^{v-1} \sin(\pi \lambda_{r',s'} s) \sin \frac{\pi \ell s}{v} \notag \\
		&= \frac{1}{2} \sum_{s=1}^{v-1} \sqbrac*{\cos(\pi (\lambda_{r',s'} - \ell / v) s) - \cos(\pi (\lambda_{r',s'} + \ell / v) s)} \notag \\
		&=
		\begin{cases*}
			\pm\frac{v}{2} & if $\ell / v = \pm\lambda_{r',s'} \pmod{\rlat}$, \\
			0 & otherwise.
		\end{cases*}
	\end{align}
	We mention that the case $\ell / v = \lambda_{r',s'} \pmod{\rlat}$ and $\ell / v = -\lambda_{r',s'} \pmod{\rlat}$, which would give $0$ for this sum, does not occur because it would require that $2 \lambda_{r',s'} \in \rlat$, hence that $us' \in 2v \ZZ$, hence $s' \in v \ZZ$.

	The contribution to $\fepfone{\mu}{r}{-1/\tau}$ from the $\ell$ satisfying $\ell / v = \lambda_{r',s'} \pmod{\rlat}$ is therefore
	\begin{equation} \label{eq:Contribution}
		-\ii \tau \frac{2}{\sqrt{uw}} \underset{\mathclap{\ell / v = \lambda_{r',s'} \ (\mathrm{mod}\ \rlat)}}{\sum_{(r',s')} \sum_{\ell = 0}^{2p-1}} (-1)^{rs'} \sin \frac{\pi r r'}{t} \ee^{-\ii \pi \ell \mu / w} \frac{\vch{u,v}{r',s'}{\tau}}{\eta(\tau)} \fdjth{\ell / v + \lat{L}}{\tau}.
	\end{equation}
	As the summands are manifestly $2p$-periodic in $\ell$, the $\ell$-sum is over a full period --- this is the reason we use \eqref{eq:Theta'S'} above instead of \eqref{eq:DefSone}.  Writing $\ell / v = r'-1-ts' = r'-1-ks' \pmod{\rlat}$, we set $\mu' = \ell /v + ks'$ so that $\mu' = r'-1 \pmod{\rlat}$ and $\mu'$ is $2w$-periodic.  The contribution \eqref{eq:Contribution} now becomes
	\begin{equation} \label{eq:Contribution'}
		-\ii \tau \frac{-2}{\sqrt{uw}} \underset{\mathclap{\mu' = r'-1 \ (\mathrm{mod}\ \rlat)}}{\sum_{(r',s')} \sum_{\mu' = 0}^{2w-1}} (-1)^{s'-1} \sin \frac{\pi r r'}{t} \ee^{\ii \pi \mu \mu' / k} \frac{\vch{u,v}{r',s'}{\tau}}{\eta(\tau)} \fdjth{\mu' - s'k + \lat{L}}{\tau},
	\end{equation}
	where we have noted that $(-1)^{rs'} \ee^{-\ii \pi s' \mu} = (-1)^{(\mu - r) s'} = -(-1)^{s'-1}$, since $\mu = r-1 \pmod{\rlat}$.

	For the contribution from the $\ell$ satisfying $\ell / v = -\lambda_{r',s'} \pmod{\rlat}$, we likewise write $\mu' = \ell /v - ks'$ so that again $\mu' = -r'+1 = r'-1 \pmod{\rlat}$ and $\mu'$ is $2w$-periodic.  This contribution now evaluates to the same form as \eqref{eq:Contribution'} but multiplied by $-1$, because of the sign in \eqref{eq:TheSSum}, and with $\djth{\mu' - s'k + \lat{L}}$ replaced by $\djth{\mu' + s'k + \lat{L}}$.  We therefore conclude that
	\begin{equation} \label{eq:NearlyThere}
		\fepfone{\mu}{r}{-1/\tau} = -\ii \tau \frac{2}{\sqrt{uw}} \underset{\mathclap{\mu' = r'-1 \ (\mathrm{mod}\ \rlat)}}{\sum_{(r',s')} \sum_{\mu' = 0}^{2w-1}} (-1)^{s'-1} \sin \frac{\pi r r'}{t} \ee^{\ii \pi \mu \mu' / k} \frac{\vch{u,v}{r',s'}{\tau}}{\eta(\tau)} \sqbrac*{\fdjth{\mu' + s'k + \lat{L}}{\tau} - \fdjth{\mu' - s'k + \lat{L}}{\tau}}.
	\end{equation}

	Suppose now that $u$ is odd.  Then, the $(r',s')$-sum over the Virasoro Kac table $\kactable{u,v}$ may be expressed as a double sum where $r'$ ranges from $1$ to $\frac{1}{2}(u-1)$ and $s'$ ranges from $1$ to $v-1$.  In this case,
	\begin{equation}
		\fepfone{\mu}{r}{-1/\tau} = -\ii \tau \frac{2}{\sqrt{uw}} \underset{\mathclap{\mu' = r'-1 \ (\mathrm{mod}\ \rlat)}}{\sum_{r'=1}^{\frac{1}{2}(u-1)} \sum_{\mu' = 0}^{2w-1}} \sin \frac{\pi r r'}{t} \ee^{\ii \pi \mu \mu' / k} \fepfone{\mu'}{r'}{\tau}.
	\end{equation}
	We can write this as a sum over the spanning set $\mathbf{B}_k$ of \cref{NotABasis} by combining the terms with $\mu \neq 0, w$ using $\epfone{\mu'}{r'} = \epfone{2w-\mu'}{r'}$, which follows from the first two identities of \eqref{eq:Wt1Symmetries}.  This gives the values of $A_{\mu'; r'}$ and the modular S-transform of the \lcnamecref{thm:atypicalmodularity} for $u$ odd.

	If $u$ is even, so $v$ is odd, then the $(r',s')$-sum may be expressed as the sum of two contributions, the first being a double sum with $r'$ ranging from $1$ to $\frac{u}{2}-1$ and $s'$ ranging from $1$ to $v-1$ while the second is a single sum with $r'$ fixed at $\frac{u}{2}$ and $s'$ ranging from $1$ to $\frac{v-1}{2}$.  The contribution from $r' < \frac{u}{2}$ is analysed as in the $u$ odd case with the same result (the upper limit of the $r'$-sum is now $\frac{u}{2}-1$).

	The analysis of the contribution from $r' = \frac{u}{2}$ is a little more intricate.  First, note that if $r$ is even, then this contribution vanishes because $\sin(\pi r r' t^{-1}) = 0$.  We may therefore assume that $r$ is odd, hence that $\mu$ is even.  We now compare the $r' = \frac{u}{2}$ contribution to \eqref{eq:NearlyThere} from a given $s'$ to that from $v-s'$.  The latter is
	\begin{align}
		&-\ii \tau \frac{2}{\sqrt{uw}} \underset{\mathclap{\mu' = u/2-1 \ (\mathrm{mod}\ \rlat)}}{\sum_{\mu' = 0}^{2w-1}} (-1)^{v-s'-1} \sin \frac{\pi r v}{2} \ee^{\ii \pi \mu \mu' / k} \frac{\vch{u,v}{u/2,v-s'}{\tau}}{\eta(\tau)} \sqbrac*{\fdjth{\mu' + (v-s')k + \lat{L}}{\tau} - \fdjth{\mu' - (v-s')k + \lat{L}}{\tau}} \notag \\
		&\mspace{20mu} = -\ii \tau \frac{2}{\sqrt{uw}} \underset{\mathclap{\mu' = u/2-1 \ (\mathrm{mod}\ \rlat)}}{\sum_{\mu' = 0}^{2w-1}} (-1)^{s'-1} \sin \frac{\pi r v}{2} \ee^{\ii \pi \mu \mu' / k} \frac{\vch{u,v}{u/2,s'}{\tau}}{\eta(\tau)} \sqbrac*{\fdjth{\mu' + w + s'k + \lat{L}}{\tau} - \fdjth{\mu' + w - s'k + \lat{L}}{\tau}} \notag \\
		&\mspace{20mu} =-\ii \tau \frac{2}{\sqrt{uw}} \underset{\mathclap{\mu' = u/2-1 \ (\mathrm{mod}\ \rlat)}}{\sum_{\mu' = 0}^{2w-1}} (-1)^{s'-1} \sin \frac{\pi r v}{2} \ee^{\ii \pi \mu \mu' / k} \frac{\vch{u,v}{u/2,s'}{\tau}}{\eta(\tau)} \sqbrac*{\fdjth{\mu' + s'k + \lat{L}}{\tau} - \fdjth{\mu' - s'k + \lat{L}}{\tau}},
	\end{align}
	where we have used the following facts: $v$ is odd, $w$ and $\mu$ are even, the theta function derivatives are $2w$-periodic, and the $\mu'$-sum is over a full period.  The contributions from $s'$ and $v-s'$ therefore coincide for $r$ odd and $r' = \frac{u}{2}$.  The $r' = \frac{u}{2}$ contribution to $\fepfone{\mu}{r}{-1/\tau}$ is then half that obtained by summing $s'$ from $1$ to $v-1$.  It therefore has exactly the same form as the generic case $r' \neq \frac{u}{2}$ analysed above except for the additional factor of $\frac{1}{2}$.
\end{proof}

An obvious consequence of this result is that the vector space spanned by the $\epfone{\mu}{r}$ and $-\ii \tau \epfone{\mu}{r}$ carries a representation of the modular group.  Recasting the character formula of \cref{prop:epfirrchar=theta} in the form
\begin{equation}
	\ch{\epfirr{\mu}{r}} = \epfone{\mu}{r} + \sum_{s=1}^{v-1} \brac*{\frac{\mu - (v-s)k}{2w} \ch{\epfrel{\mu+sk}{r,s}} - \frac{\mu + (v-s)k}{2w} \ch{\epfrel{\mu-sk}{r,s}}}
\end{equation}
and recalling \cref{rem:D=E+L}, we conclude that the direct sum of this vector space and the span of the standard characters $\ch{\epfrel{\mu}{r,s}}$, with $\mu \in \lat{L}' / \lat{L}$ and $(r,s) \in \kactable{u,v}$, contains the characters of all the irreducible $\epfvoa{k}$-modules and carries a $\SLG{PSL}{2;\ZZ}$-representation.

\begin{corollary}
	Given \cref{ass:cat,ass:fusion}, it follows that the characters of the irreducible $\epfvoa{k}$-modules are modular: they generate a representation of $\SLG{PSL}{2;\ZZ}$ of dimension $p(u-1)(v-1) + 2 \dim \vvmf{k} < \infty$.
\end{corollary}

\begin{conjecture} \label{conj}
	The \voa{} $\epfvoa{k}$ is $C_2$-cofinite.
\end{conjecture}

\begin{remark}
	This $C_2$-cofiniteness is known for $\epfvoa{-1/2}$ because is it isomorphic to the triplet $\tripvoa{1,2}$ \cite{CarNon06}, $\epfvoa{-4/3}$ because it is isomorphic to a $\ZZ_2$-orbifold of $\tripvoa{1,3}$ \cite{CarNon06,MiyC2C15}, and $\epfvoa{-2/3}$ because it is isomorphic to a $\ZZ_2 \oplus \ZZ_2$-orbifold of $\stripvoa{1,3}$ \cite{AdaN=109} (assuming that the main result of \cite{MiyC2C15} also holds for appropriate \svoas{}).
\end{remark}

\flushleft

\end{document}